\documentclass{amsart}
\usepackage{amsmath, amssymb, graphicx, amsthm}
\usepackage{epstopdf}

\numberwithin{equation}{section}

\theoremstyle{plain}
\newtheorem{prop}{Proposition}[section]
\newtheorem{conj}[prop]{Conjecture}
\newtheorem{coro}[prop]{Corollary}

\newtheorem{lemm}[prop]{Lemma}
\newtheorem{ques}[prop]{Question}
\newtheorem{thrm}{Theorem}

\theoremstyle{definition}
\newtheorem{defi}[prop]{Definition}
\newtheorem{rema*}{Remark}
\newtheorem{rema}[prop]{Remark}
\newtheorem{exam}[prop]{Example}

\renewcommand\aa{a}

\newcommand\bb{b}

\newcommand\BP[1]{B_{#1}^{\scriptscriptstyle\pmb+}}

\newcommand\cc{c}
\newcommand\CC{C}
\newcommand\clp[1]{[#1]}
\newcommand\compl{\mathrm{compl}_{\rev}}
\newcommand\cp{\backslash}
\newcommand\cpR{\backslash_{{}_\RR}}

\newcommand\Diag[1]{D_{#1}}
\newcommand\dist{\mathrm{dist}}
\newcommand\distR{\mathrm{dist}}
\newcommand\distRrev{\mathrm{dist}_{\rev}\!}

\newcommand\DL{D_{\scriptscriptstyle \!L}}
\newcommand\DR{D_{\scriptscriptstyle \!R}}
\newcommand\DRL{D}

\newcommand\ee{e}
\newcommand\eqp{\equiv^{\scriptscriptstyle+}}
\newcommand\eqR{\equiv_\RR}
\newcommand\eqmR{\equiv_\RR^{\scriptscriptstyle-}}
\newcommand\eqpR{\equiv_\RR^{\scriptscriptstyle+}}
\newcommand\ew{\varepsilon}

\newcommand\ff{f}
\newcommand\FF{F}

\let\ge=\geqslant

\newcommand\GG{G}
\newcommand\Gr[2]{\langle#1\mid\nobreak#2\rangle}

\newcommand\ie{{\it i.e.}}
\newcommand\ii{i}
\newcommand\inv{^{-1}}

\newcommand\jj{j}

\newcommand\kk{k}
\newcommand\KKK{\mathcal{K}}

\newcommand\lb{\lambda}
\let\le=\leqslant

\newcommand\mm{m}
\newcommand\MM{M}
\newcommand\Mon[2]{\langle#1\mid\nobreak#2\rangle^{\scriptscriptstyle+}}

\newcommand\Name[3]{\{#1,#2\}_{#3}}

\newcommand\nn{n}

\newcommand\NNNN{\mathbb{N}}
\newcommand\nno{{n-1}}
\newcommand\NL{N_{\scriptscriptstyle \!L}}
\newcommand\NR{N_{\scriptscriptstyle \!R}}
\newcommand\NRL{N}

\newcommand\pp{p}

\newcommand\qq{q}

\newcommand\resp{{\it resp.} }
\newcommand\rev{\curvearrowright}
\newcommand\revd{\circlearrowright}
\newcommand\revdR{ \revd_\RR}
\newcommand\revl{\mathrel{\raisebox{5pt}{\rotatebox{180}{$\curvearrowleft$}}}}
\newcommand\revlR{\revl_{\!\RR}}
\newcommand\revm{\rightsquigarrow}
\newcommand\revmR{\revm_\RR}
\newcommand\revR{ \curvearrowright_\RR}
\newcommand\rr{r}
\newcommand\RR{\mathcal{R}}
\newcommand\RRh{\widehat\RR}

\newcommand\Sep{\Sigma}
\newcommand\SEP[3]{\Sigma_{#1,#2,#3}}

\newcommand\sig[1]{\sigma_{\!#1}^{\relax}}
\newcommand\siginv[1]{\sigma_{\!#1}^{-1}}
\newcommand\sigg[2]{\sigma_{\!#1}^{#2}}
\newcommand\sName[3]{\scriptstyle\{\!#1,#2\!\}_{\!#3}}
\renewcommand\ss{s}
\renewcommand{\SS}{S}
\newcommand{\SSS}{\mathcal{S}}
\newcommand{\SSSh}{\widehat\SSS}
\newcommand\sss{s'}
\newcommand\ssss{s''}
\newcommand\su{\underline{\boldsymbol{u}}}
\newcommand\suu{\su'}
\newcommand\sv{\underline{\boldsymbol{v}}}
\newcommand\svv{\sv'}
\newcommand\sw{\underline{\boldsymbol{w}}}
\newcommand\sww{\sw'}
\newcommand\swww{\sw''}

\newcommand\tta{\mathtt{a}}
\newcommand\ttA{\mathtt{A}}
\newcommand\ttb{\mathtt{b}}
\newcommand\ttB{\mathtt{B}}
\newcommand\ttc{\mathtt{c}}
\newcommand\ttC{\mathtt{C}}
\newcommand\ttd{\mathtt{d}}
\newcommand\ttD{\mathtt{D}}
\newcommand\tte{\mathtt{e}}
\newcommand\ttf{\mathtt{f}}

\newcommand\uu{u}
\newcommand\uuu{u'}
\newcommand\uuuu{u''}

\def\VR(#1,#2){\vrule width0pt height#1mm depth#2mm}
\newcommand\vv{v}

\newcommand\vvv{v'}
\newcommand\vvvv{v''}

\newcommand\ww{w}
\newcommand\www{w'}

\newcommand\xx{x}
\newcommand\XX{X}
\newcommand\xxx{\xx'}

\newcommand\yy{y}

\newcommand\zz{z}
\newcommand\ZZZZ{\mathbb{Z}}
\begin{document}

\hfill{\tiny 2009-12}

\author{Patrick DEHORNOY}

\address{Laboratoire de Math\'ematiques Nicolas Oresme, UMR 6139 CNRS,
Universit\'e de Caen, 14032 Caen, France}
\email{dehornoy@math.unicaen.fr}
\urladdr{//www.math.unicaen.fr/\!\hbox{$\sim$}dehornoy}

\title{The subword reversing method}

\keywords{semigroup presentation, van Kampen diagram, rewrite system, cancellativity, word problem, Garside monoid, group of fractions, monoid embeddability}

\subjclass{20B30, 20F55, 20F36}

\footnote{Work partially supported by the ANR grant ANR-08-BLAN-0269-02}

\begin{abstract}
We summarize the main known results involving subword reversing, a method of semigroup theory for constructing van Kampen diagrams by referring to a preferred direction. In good cases, the method provides a powerful tool for investigating presented (semi)groups. In particular, it leads to cancellativity and embeddability criteria for monoids and to efficient solutions for the word problem of monoids and groups of fractions. 
\end{abstract}

\maketitle

Subword reversing is a combinatorial method for investigating presented semigroup.  It has been developed in various contexts and the results are scattered in different sources~\cite[...]{Dfa, Dff, Dfx, Dgc, Dgp, Dhg, Dhx}. This text is a survey that discusses the main aspects of the method, its range, its uses, and its efficiency. The emphasis is put on the exportable applications rather than on the internal technicalities, for which we refer to literature. New examples and open questions are mentioned, as well as a few new results. Excepted in the cases where no reference is available, proofs are sketched, or just omitted.

\subsection*{General context and main results}

As is well known, working with a semigroup or a group presentation is usually very difficult, and most problems are undecidable in the general case. Subword reversing is one of the few methods that can be used to investigate a presented semigroup, possibly a presented group. The specificity of the method is that, in order to solve the word problem of a presented semigroup, or, equivalently, construct a van Kampen diagram for a pair of initially given words, one directly compares the words one to the other instead of separately reducing each of them to some normal form, as in standard approaches like Knuth--Bendix algorithm or Gr\"obner--Shirshov bases (see Figure~\ref{F:Comparison}).

\begin{figure}[t]
$$\begin{picture}(91,25)(0,1)
\put(0,0.8){\includegraphics{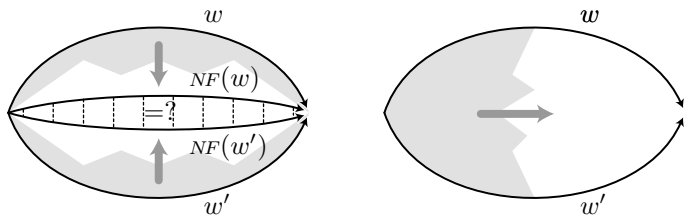}}
\put(27,24.5){$\ww$}
\put(25,16){${\scriptstyle N\!F}(\ww)$}
\put(25,7){${\scriptstyle N\!F}(\www)$}
\put(27,-1){$\www$}
\put(77,24.5){$\ww$}
\put(77,24.5){$\ww$}
\put(77,-1){$\www$}
\put(19,11.5){$=$?}
\end{picture}$$
\caption{\sf Solving the word problem of a presented semigroup: to compare two words~$\ww$ and~$\www$, contrary to methods based on rewrite systems, which separately reduce $\ww$ and~$\www$ to some distinguished equivalent words~${\scriptstyle N\!F}(\ww)$, ${\scriptstyle N\!F}(\www)$ and check the equality of the latter (left diagram), word reversing (right diagram) appeals to no normal form and tries to directly construct a van Kampen diagram by reading the letters from left to right.}
\label{F:Comparison}
\end{figure}

Every semigroup presentation is in principle eligible for subword reversing, but the method leads to useful results only when some condition called completeness is satisfied. The good news is that the completeness condition is satisfied in a number of nontrivial cases and that, even if it is not initially satisfied, it can be satisfied once a certain completion procedure has been performed. 

The general philosophy is that, whenever the completeness condition is fulfilled, some properties of the considered semigroup can be read from the presentation easily. Typically, when a presentation is complete, it is sufficient that the presentation contains no obvious obstruction to left-cancellativity, namely no relation of the form $\ss\vv = \ss \vvv$ with $\vv \not= \vvv$, to be sure that the presented semigroup does admit left-cancellation. Combined with a completeness criterion (several exist), this leads to practical, easy to use, cancellativity criteria, such as the following one.

\begin{thrm} [a criterion for left-cancellatibility]
\label{T:Intro1}
Assume that a semigroup (or a monoid)~$\MM$ admits a presentation~$(\SSS, \RR)$ satisfying the following conditions:

$(i)$ The set~$\RR$ contains no relation $\ss\vv = \ss \vvv$ with $\ss$ in~$\SSS$ and $\vv \not= \vvv$;

$(ii)$ There exists~$\lambda : \MM \to \NNNN$ satisfying $\lambda(\xx\yy) \ge \lambda(\xx) + \lambda(\yy)$ for all~$\xx, \yy$ in~$\MM$ and $\lambda(\ss) \ge 1$ for each~$\ss$ in~$\SSS$; 

$(iii)$ The right cube condition holds for each triple in~$\SSS^3$---see Definition~\ref{D:Cube}.

\noindent Then $\MM$ admits left-cancellation.
\end{thrm}

Similarly, if, for some generators~$\ss, \sss$, we have in the list of relations several relations of the form $\ss \vvv = \sss \vv$, then, in general, in the corresponding semigroup, the elements~$\ss$ and~$\sss$ admit no least common right-multiple (right-lcm), \ie, no common right-multiple of which every common right-multiple of~$\ss$ and~$\sss$ is a right-multiple. In the case of a complete presentation, it is sufficient that the above obstruction does not occur to be sure that the monoid does admit right-lcm's.

\begin{thrm} [a criterion for the existence of right lcm's]
\label{T:Intro2}
Assume that a semigroup (or a monoid)~$\MM$ admits a presentation~$(\SSS, \RR)$ satisfying the following conditions:

$(i)$ For all~$\ss, \sss$ in~$\SSS$, there is at most one relation of the form $\ss\vvv = \sss \vv$ in~$\RR$;

$(ii)$ There exists~$\lambda : \MM \to \NNNN$ satisfying $\lambda(\xx\yy) \ge \lambda(\xx) + \lambda(\yy)$ for all~$\xx, \yy$ in~$\MM$ and $\lambda(\ss) \ge 1$ for each~$\ss$ in~$\SSS$; 

$(iii)$ The right cube condition holds for each triple in~$\SSS^3$---see Definition~\ref{D:Cube}.

\noindent Then any two elements of~$\MM$ that admit a common right-multiple admit a least common right-multiple.
\end{thrm}

On the other hand, subword reversing is also an algorithmic process, and it can be used to recognize divisors or solve the word problem of the semigroup, and possibly of its enveloping group. Taking for granted the definition of the reversing relation~$\revR$ (see Definition~\ref{D:Reversing}) we have in particular: 

\begin{thrm} [a solution of the word problem]
\label{T:Intro3}
Assume that a group~$\GG$ admits a semigroup presentation\footnote{\ie, all relations are of the form $\vv = \vvv$ with $\vv, \vvv$ nonempty and containing no inverse of the generators} $(\SSS, \RR)$ satisfying the following conditions:

$(i)$ The set~$\RR$ contains no relation $\ss\vv = \ss \vvv$ or $\vv \ss = \vvv \ss$ with $\ss$ in~$\SSS$ and $\vv \not= \vvv$;

$(iii)$ There exists~$\lambda : \Mon\SSS\RR \to \NNNN$ satisfying $\lambda(\xx\yy) \ge \lambda(\xx) + \lambda(\yy)$ for all~$\xx, \yy$ in~$\Mon\SSS\RR$ and $\lambda(\ss) \ge 1$ for each~$\ss$ in~$\SSS$; 

$(iii)$ For all~$\ss, \sss$ in~$\SSS$, there is at most one relation of the form $\ss\vvv = \sss \vv$ in~$\RR$;

$(iv)$ The left and right cube conditions hold for each triple in~$\SSS^3$; 

$(v)$ There exists a set of words in the alphabet~$\SSS$, say~$\SSSh$, that includes~$\SSS$ and is such  that, for all~$\uu, \uuu$ in~$\SSSh$, there exist $\vv, \vvv$ in~$\SSSh$ satisfying $\uu\inv \uuu \revR \vvv \vv\inv$.

\noindent Then a word~$\sw$ in the alphabet~$\SSS \cup \SSS\inv$ represents~$1$ in~$\GG$ if and only if $\vv\inv \vvv \revR\nobreak \ew$ holds, where $\vv$ and~$\vvv$ are the (unique) words in the alphabet~$\SSS$ that satisfy $\sw \revR \nobreak \vvv \vv\inv$.
\end{thrm}

The above statements\footnote{actually, these are rather templates, as several variants exist; in particular, it is not necessary that the same presentation is used to establish the various hypotheses, see Remark~\ref{R:Various}} look quite technical, and one may wonder whether any presentation satisfies the many involved requirements. Actually, such presentations do exist, and there is even a number of them. Indeed, every Artin--Tits presentation is eligible and, more generally, every Garside group admits presentations that satisfy the above conditions. On the other hand, it is of course easy to construct examples that do not satisfy the conditions, and we do not claim that subword reversing is of universal interest. What we do claim is that, when one is to address an unknown semigroup presentation, it is always worth trying reversing. Let us mention that, in some cases such as the above-mentioned Artin--Tits presentations, reversing (or essentially equivalent methods) is the only method known so far for establishing cancellativity.

Further applications of subword reversing will be mentioned. As a general rule, the method is well fitted to work with the so-called Garside monoids and groups. In particular, it is eligible to compute least common multiples, greatest common divisors, and the derived unique normal forms (``greedy normal forms'').

\subsection*{Historical comments}

Subword reversing is, in some sense, the most obvious and elementary  approach for effectively constructing van Kampen diagrams (see Section~\ref{S:Description} below), and it could have been introduced in the early years of the twentieth century. However it seems it was not considered until much later.

A precursor of subword reversing can be found in Garside's approach to Artin's braid groups~\cite{Gar} and in the subsequent extension to spherical Artin--Tits groups by Brieskorn and Saito~\cite{BrS}: in particular, Theorem~H of~\cite{GarT} and~\cite{Gar} amounts to saying that Artin's presentation is complete with respect to subword reversing\footnote{it may be interesting to mention that, in~Ê\cite{GarT}, the principle of the proof of Theorem~H is attributed by F.A.\,Garside to his advisor G.\,Higman}. However, the viewpoint is slightly different from what will be developed below, and reversing remains implicit in these sources.

It seems that subword reversing in its current form was first explicitly considered in~\cite{Dez, Dfb} with the specific aim of investigating the so-called geometry monoid of self-distributivity and establishing cancellativity results. Soon after, the  eligibility of Artin's braid monoids---which turn out to be projections of the self-distributivity monoid---was observed~\cite{Dfa, Dff}, and the connection with Garside's approach became clear. At the same time, again in the case of braid monoids and Artin--Tits monoids, the approach of Tatsuoka in~\cite{Tat}, and, slightly later, that of Corran in~\cite{Cor}, are closely connected. All these approaches are essentially equivalent and equally relevant in the case of presentations that define monoids in which least common multiples exist (``complete complemented presentations'' according to the terminology of Section~\ref{S:Complete} below, ``chainable presentations'' according to the terminology of~\cite{Cor}). However, it seems that only subword reversing is suitable for an extension to more general cases~\cite{Dgp}.

As already illustrated in Figure~\ref{F:Comparison}, there seems to be no connection between subword reversing and the other general algorithmic methods relevant for (semi)groups, because the latter rely on a totally different approach for solving the word problem. If subword reversing is to be compared with another existing method, it is Dehn's algorithm and small cancellation techniques that seem the closest: all have in common that a van Kampen diagram is built by using a convenient fragment of the boundary at each step. However, in the case of subword reversing, the boundary is defined dynamically, resulting in a quadratic complexity rather than in a linear complexity.

At another level, the completion procedure involved as a preprocessing step in the subword reversing method turns out to have very little in common with the one involved in the Gr\"obner base approach as adapted to the context of presented semigroups~Ê\cite{Aut}.

\subsection*{Organization of this text}

In Section~\ref{S:Description}, we describe subword reversing as a particular strategy for constructing van Kampen diagrams. In Section~\ref{S:Range}, we analyze the range of the method, \ie, we state the additional conditions under which reversing is possibly useful, namely those guaranteeing the so-called completeness property. Then, in Section~\ref{S:Uses}, we list some results that can be obtained---in good cases---using subword reversing, including a cancellativity criterion that is maybe the most striking application of the method. Finally, in Section~\ref{S:Efficiency}, we address the question of whether subword reversing, when eligible, leads to efficient algorithms, in particular in terms of solution of the word problem and of isoperimetric inequalities. 

The main new results proved in this text are those of Section~\ref{S:Mixed} (Propositions~\ref{P:Embeddability} and~\ref{P:MixedBis}) about mixed reversing and Section~\ref{S:Lower} (Proposition~\ref{P:Optimal}) about the optimality of reversing and its applications to the combinatorial distance between braid words.

\subsection*{Acknowledgment}

The author thanks J\'er\'emy Chamboredon for his help in preparing the final version of this text.

\section{Subword reversing: description}
\label{S:Description}

Subword reversing can equivalently be described as a syntactic transformation on words, or as a strategy for constructing van Kampen diagrams in the context of presented semigroups or monoids. Here we give both descriptions, starting with the latter, which is more visual and concrete.

\subsection{Van Kampen diagrams}

Hereafter we always work with monoids rather than with arbitrary semigroups, \ie, we always assume that our semigroups contain a unit element, usually denoted~$1$. This option is convenient, but unessential.

Assume that $(\SSS, \RR)$ is a semigroup presentation, \ie, $\SSS$ is a (finite or infinite) nonempty set and $\RR$ is a (finite or infinite) family of pairs of nonempty words in the alphabet~$\SSS$, usually called \emph{relations}. We denote by $\Mon\SSS\RR$ the monoid presented by $(\SSS, \RR)$, \ie, the quotient-monoid $\SSS^*{/}{\eqp_\RR}$ where $\SSS^*$ denotes the free monoid of all words in the alphabet~$\SSS$ and $\eqp_\RR$ denotes the least congruence on~$\SSS^*$ (multiplication-compatible equivalence relation) that includes~$\RR$. As is well-known, two words~$\ww, \www$ of~$\SSS^*$ are $\RR$-equivalent, \ie, connected under~$\eqpR$, if and only if there exists an $\RR$-derivation from~$\ww$ to~$\www$, defined to be a finite sequence of words~$(\ww_0, \, ... \, , \ww_\pp)$ such that $\ww_0$ is~$\ww$, $\ww_\pp$ is~$\www$, and, for each~$\ii$, there exists $\{\vv, \vvv\}$ in~$\RR$ and $\uu, \uuu$ in~$\SSS^*$ satisfing $\{\ww_\ii, \ww_{\ii+1}\} = \{\uu \vv \uuu, \uu \vvv \uuu\}$, \ie, $\ww_{\ii+1}$ is obtained from~$\ww_\ii$ by substituting some subword that occurs in a relation of~$\RR$ with the other element of that relation. 

In the above context, by construction, for each relation $\{\vv, \vvv\}$ of~$\RR$, the elements of the monoid~$\Mon\SSS\RR$ represented by~$\vv$ and by~$\vvv$ are equal. Owing to this fact, it is customary to denote the relation $\{\vv, \vvv\}$ as $\vv = \vvv$.

An $\RR$-derivation can be nicely visualized using a van Kampen diagram.  A \emph{$(\SSS,\RR)$-van Kampen diagram} for a pair of words~$(\ww, \www)$ is a planar oriented graph with a unique source vertex and a unique sink vertex and  edges labeled by letters of~$\SSS$, so that the labels of each face correspond to a relation of~$\RR$ and the labels of the bounding paths form the words~$\ww$ and~$\www$, respectively. 

\begin{lemm}[folklore]
\label{L:Derivation}
If $(\SSS, \RR)$ is a semigroup presentation, then two words~$\ww$ and~$\www$ of~$\SSS^*$ are $\RR$-equivalent if and only if there exists an $(\SSS, \RR)$-van Kampen diagram for~$(\ww, \www)$.
\end{lemm}

\begin{proof}[Proof (sketch)]
If $(\ww_0, ..., \ww_\pp)$ is an $\RR$-derivation from~$\ww$ to~$\www$, then drawing paths labeled with the successive words~$\ww_\ii$ one below the other and identifying the unchanged letters yields a van Kampen diagram for~$(\ww, \www)$. Conversely, if $\KKK$ is a van Kampen diagram for~$(\ww, \www)$, one obtains an $\RR$-derivation from~$\ww$ to~$\www$ by enumerating the labels in a sequence of paths from the source of~$\KKK$ to its sink that differ by one face at a time.
\end{proof}

\begin{exam}
\label{X:Preferred}
In the sequel we shall often consider the presented monoid
$$\MM = \Mon{\tta, \ttb, \ttc, \ttd}
{\tta\ttb = \ttb\ttc = \ttc\tta,
\ttb\tta = \ttd\ttb = \tta\ttd}.$$
Then $\tta\ttc\tta\tta\tta$ and $\ttc\ttd\ttb\ttb\ttb$ represent the same element of~$\MM$ as we have
$$\tta\ttc\tta\tta\tta \eqp{} \tta\ttb\ttc\tta\tta
\eqp{} \tta\ttb\ttb\ttc\tta
\eqp{} \ttc\tta\ttb\ttc\tta
\eqp{} \ttc\tta\ttb\tta\ttb
\eqp{} \ttc\tta\ttd\ttb\ttb
\eqp{} \ttc\ttd\ttb\ttb\ttb.$$
A van Kampen diagram corresponding to this derivation is displayed in Figure~\ref{F:Kampen}.
\end{exam}

\begin{figure}[htb]
$$\begin{picture}(58,28)(0,3)
\put(0,0.8){\includegraphics{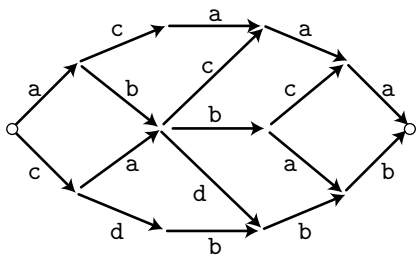}}
\put(3,22){$\tta$}
\put(3,11){$\ttc$}
\put(14,30){$\ttc$}
\put(14,3){$\ttd$}
\put(16,22){$\ttb$}
\put(16,12){$\tta$}
\put(26,25){$\ttc$}
\put(25,8){$\ttd$}
\put(27,32){$\tta$}
\put(27,18.5){$\ttb$}
\put(27,1){$\ttb$}
\put(37,22){$\ttc$}
\put(37,12){$\tta$}
\put(39,30){$\tta$}
\put(39,3){$\ttb$}
\put(50,22){$\tta$}
\put(50,11){$\ttb$}
\end{picture}$$
\caption{\sf A van Kampen diagram for the derivation of Example~\ref{X:Preferred}: the labels of the top path form the word~$\tta\ttc\tta\tta\tta$, those of the bottom path form $\ttc\ttd\ttb\ttb\ttb$, and the diagram is tessellated 
by tiles that correspond to relations.}
\label{F:Kampen}
\end{figure}

\subsection{A strategy for building van Kampen diagrams}

Assuming that $(\SSS, \RR)$ is a semigroup presentation, we address the question of effectively building a van Kampen diagram for a pair of words~$(\ww, \www)$. Of course, such a diagram may exist only if $\ww$ and $\www$ are $\RR$-equivalent, so an algorithmic solution to the current question has to include a solution for the word problem of~$(\SSS, \RR)$, \ie, a method for deciding whether $\ww$ and $\www$ are $\RR$-equivalent. Pictorially, our problem consists in drawing from a common origin two paths labeled~$\ww$ and~$\www$ and tessellating the space between these paths with tiles corresponding to the relations of~$\RR$.

In this context, subword reversing is the most straightforward strategy, namely starting from the two edges~$\ss, \sss$ that originate in the source vertex, choosing one relation $\ss ... = \sss ...$ in~$\RR$ and iterating the process with the next vertices. So, if the paths $\ww$ and $\www$ have an overall orientation from left to right (as in Figure~\ref{F:Kampen}), subword reversing can be called the ``left strategy'' as it corresponds to proceeding from left to right, namely

- looking at a (leftmost) pending pattern 
\VR(7,3)\begin{picture}(8,4)(0,3)
\put(0,0.5){\includegraphics[scale=0.7]{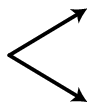}}
\put(2,0){$\ss$}
\put(2,6){$\sss$}
\end{picture},

- choosing a relation $\ss\vvv = \sss\vv$ of~$\RR$, closing this pattern into
\VR(5,5)\begin{picture}(12,4)(0,3)
\put(0,0.5){\includegraphics[scale=0.7]{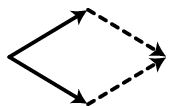}}
\put(2,0){$\ss$}
\put(2,6){$\sss$}
\put(9,0){$\vvv$}
\put(9,6){$\vv$}
\end{picture},
and repeat.\\
As it stands, the approach seems naive, and, clearly, it cannot be successful in every case. Several obstructions may occur. In particular, one gets stuck if, at some step, there is no eligible relation $\ss ... = \sss ... $ in~$\RR$. Also, the process may never terminate, or it may terminate but  boundary words be longer than $\ww, \www$, \ie, in order to close the diagram one has to extend the initial words---which is the best we can hope for if the initial words $\ww, \www$ are not $\RR$-equivalent. Also, we observe that the strategy need not be deterministic: if there exist letters~$\ss, \sss$ such that $\RR$ contains several relations $\ss ... = \sss ... $, each of these is eligible and there are several ways of performing the process.

\begin{exam}
(See Figure~\ref{F:RevStrat}.) With the presentation and the words of Example~\ref{X:Preferred}, starting from two diverging paths labeled $\tta\ttc\tta\tta\tta$ and $\ttc\ttd\ttb\ttb\ttb$, we first close the left open $(\tta, \ttc)$-pattern using the relation $\tta\ttb = \ttc\tta$. Then we have two open patterns, namely $(\tta, \ttd)$ (bottom) and $(\ttc, \ttb)$ (top). If we choose the former, we can close it using the relation $\tta\ttd = \ttd\ttb$. In this way, we find an open pattern consisting of two diverging $\ttb$-labeled edges: we can see it as a special open pattern, which can be closed using the trivial relation $\ttb = \ttb$, \ie, adding empty words---represented by dotted lines on the picture. Continuing similarly, we arrive after five steps at a diagram in which the only open pattern is $(\ttc, \ttd)$. Here we are stuck, because there is no relation $\ttc... = \ttd...$ in our list of relations. So, in this case, the reversing strategy fails: we know that there exists a van Kampen diagram (for instance, the one of Figure~\ref{F:Kampen}), but we fail to find this one or any other one using our attempted strategy. When we compare with Figure~\ref{F:Kampen}, we see that, in order to proceed and re-obtain the previous van Kampen diagram, we ought to split the open pattern into two open patterns by inserting a new, intermediate $\ttb$-labeled edge, which is precisely what our strategy tries to avoid.
\end{exam}

\begin{figure}[htb]
\begin{picture}(118,70)(0,-1)
\put(0,-0.5){\includegraphics{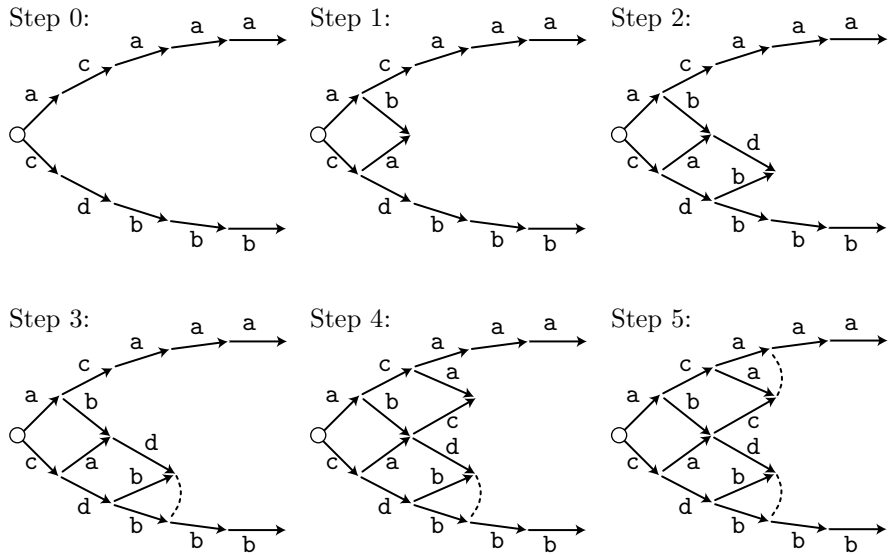}}
\put(0,68){Step 0:}
\put(2,58){$\tta$}
\put(9,62){$\ttc$}
\put(16,65){$\tta$}
\put(24,67){$\tta$}
\put(31,67.5){$\tta$}
\put(2,49){$\ttc$}
\put(9,43){$\ttd$}
\put(16,40.5){$\ttb$}
\put(24,38.5){$\ttb$}
\put(31,38){$\ttb$}
\put(40,68){Step 1:}
\put(42,58){$\tta$}
\put(49,62){$\ttc$}
\put(56,65){$\tta$}
\put(64,67){$\tta$}
\put(71,67.5){$\tta$}
\put(42,49){$\ttc$}
\put(49,43){$\ttd$}
\put(56,40.5){$\ttb$}
\put(64,38.5){$\ttb$}
\put(71,38){$\ttb$}
\put(50,57){$\ttb$}
\put(50,49){$\tta$}
\put(80,68){Step 2:}
\put(82,58){$\tta$}
\put(89,62){$\ttc$}
\put(96,65){$\tta$}
\put(104,67){$\tta$}
\put(111,67.5){$\tta$}
\put(82,49){$\ttc$}
\put(89,43){$\ttd$}
\put(96,40.5){$\ttb$}
\put(104,38.5){$\ttb$}
\put(111,38){$\ttb$}
\put(90,57){$\ttb$}
\put(90,49){$\tta$}
\put(98,51.5){$\ttd$}
\put(96,47){$\ttb$}
\put(0,28){Step 3:}
\put(2,18){$\tta$}
\put(9,22){$\ttc$}
\put(16,25){$\tta$}
\put(24,27){$\tta$}
\put(31,27.5){$\tta$}
\put(2,9){$\ttc$}
\put(9,3){$\ttd$}
\put(16,0.5){$\ttb$}
\put(24,-1.5){$\ttb$}
\put(31,-2){$\ttb$}
\put(10,17){$\ttb$}
\put(10,9){$\tta$}
\put(18,11.5){$\ttd$}
\put(16,7){$\ttb$}
\put(40,28){Step 4:}
\put(42,18){$\tta$}
\put(49,22){$\ttc$}
\put(56,25){$\tta$}
\put(64,27){$\tta$}
\put(71,27.5){$\tta$}
\put(42,9){$\ttc$}
\put(49,3){$\ttd$}
\put(56,0.5){$\ttb$}
\put(64,-1.5){$\ttb$}
\put(71,-2){$\ttb$}
\put(50,17){$\ttb$}
\put(50,9){$\tta$}
\put(58,21){$\tta$}
\put(58,14.5){$\ttc$}
\put(58,11){$\ttd$}
\put(56,7){$\ttb$}
\put(80,28){Step 5:}
\put(82,18){$\tta$}
\put(89,22){$\ttc$}
\put(96,25){$\tta$}
\put(104,27){$\tta$}
\put(111,27.5){$\tta$}
\put(82,9){$\ttc$}
\put(89,3){$\ttd$}
\put(96,0.5){$\ttb$}
\put(104,-1.5){$\ttb$}
\put(111,-2){$\ttb$}
\put(90,17){$\ttb$}
\put(90,9){$\tta$}
\put(98,21){$\tta$}
\put(98,14.5){$\ttc$}
\put(98,11){$\ttd$}
\put(96,7){$\ttb$}
\end{picture}
\caption{\sf Trying to build a van Kampen diagram for the words of Example~\ref{X:Preferred} using the reversing strategy; here the strategy fails since one gets stuck at Step~$5$.}
\label{F:RevStrat}
\end{figure}

It will be convenient to standardize the diagrams such as those of Figure~\ref{F:RevStrat} so that they only contain vertical and horizontal edges,  plus dotted arcs connecting vertices that are to be identified in order to (possibly) obtain an actual van Kampen diagram. Such standardized diagrams will be called \emph{reversing diagrams} in the sequel. For instance, the reversing diagram corresponding to the final step in Figure~\ref{F:RevStrat} is displayed in Figure~\ref{F:Reversing}.
\goodbreak

\begin{figure}[htb]
$$\begin{picture}(52,31)(0,4)
\put(1.5,0.5){\includegraphics{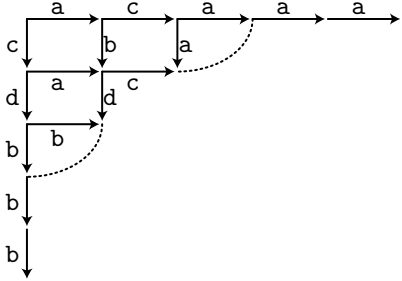}}
\put(0,3){$\ttb$}
\put(0,10){$\ttb$}
\put(0,17){$\ttb$}
\put(0,24){$\ttd$}
\put(0,31){$\ttc$}
\put(6,36){$\tta$}
\put(16,36){$\ttc$}
\put(26,36){$\tta$}
\put(36,36){$\tta$}
\put(46,36){$\tta$}
\put(13,31){$\ttb$}
\put(23,31){$\tta$}
\put(6,26){$\tta$}
\put(16,26){$\ttc$}
\put(13,24){$\ttd$}
\put(6,18.5){$\ttb$}
\end{picture}$$
\caption{\sf Reversing diagram associated with the last diagram in Figure~\ref{F:RevStrat}: the only difference is that we insist that all edges are horizontal oriented to the right or vertical oriented to bottom. Again we are stuck as there is no relation $\ttc ... = \ttd ... $ in the presentation.}
\label{F:Reversing}
\end{figure}

In this way, all tiles in a reversing diagram are obtained in a uniform way, namely by closing 
\VR(10,5)\begin{picture}(15,0)(-3,5)
\put(0,0){\includegraphics{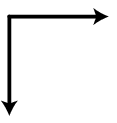}}
\put(-2,5){$\ss$}
\put(5,12){$\sss$}
\end{picture}
into
\VR(10,5)\begin{picture}(18,0)(-3,5)
\put(0,0){\includegraphics{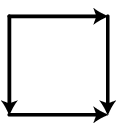}}
\put(-2,5){$\ss$}
\put(5,12){$\sss$}
\put(12,5){$\vv$}
\put(5,-2.5){$\vvv$}
\end{picture}
where $\ss\vvv=\sss\vv$ is a relation of~$\RR$, or, more accurately, in order to take possible dotted lines into account, closing 
\VR(7,5)\begin{picture}(15,0)(-3,5)
\put(0,0){\includegraphics{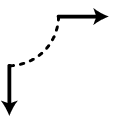}}
\put(-2,3){$\ss$}
\put(7,12){$\sss$}
\end{picture}
into
\VR(6,8)\begin{picture}(18,0)(-3,5)
\put(0,0){\includegraphics{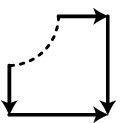}}
\put(-2,3){$\ss$}
\put(7,12){$\sss$}
\put(12,5){$\vv$}
\put(5,-2.5){$\vvv$}
\end{picture},
 including the degenerate case of
\VR(6,5)\begin{picture}(15,0)(-3,5)
\put(0,0){\includegraphics{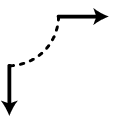}}
\put(-2,3){$\ss$}
\put(7,12){$\ss$}
\end{picture}
being closed into
\VR(6,8)\begin{picture}(15,0)(-3,5)
\put(0,0){\includegraphics{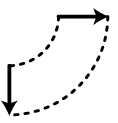}}
\put(-2,3){$\ss$}
\put(7,12){$\ss$}
\end{picture}.
\goodbreak
It should then be clear that the construction may be applied to any pair of initial paths (or, more generally, to any staircase consisting of alternating horizontal and vertical paths), and that the following three behaviours are {\it a priori} possible:

- $(i)$ either one gets stuck with a pair of letters~$(\ss, \sss)$ such that $\RR$ contains no relation $\ss... = \sss...$,

- $(ii)$ or the process continues for ever,

- $(iii)$ or the process leads in finitely many steps to a diagram of the form
\VR(6,8)\begin{picture}(23,10)(-3,5)
\put(0.5,0){\includegraphics{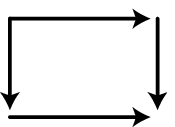}}
\put(-2,5){$\ww$}
\put(7,12){$\www$}
\put(17,5){$\vv$}
\put(7,-2){$\vvv$}
\end{picture}.

In case~$(iii)$, up to identifying the vertices that are connected by a dotted line, the reversing diagram projects to a van Kampen diagram witnessing that the words $\ww\vvv$ and $\www \vv$ are $\RR$-equivalent. Therefore, these words represent a common right-multiple of the elements of the monoid~$\Mon\SSS\RR$ represented by~$\ww$ and~$\www$, and the reversing process can be viewed as a method not only for proving the equivalence of two words, but also, more generally, for constructing common right-multiples. The case when the words~$\vv$ and~$\vvv$ are empty, \ie, when the process terminates without introducing additional edges, is the case when the method gives an actual van Kampen diagram for~$(\ww, \www)$.

\subsection{Syntactic description}

The reversing method can easily be described using words in a symmetrized alphabet~$\SSS \cup \SSS\inv$, where $\SSS\inv$ is a formal copy of~$\SSS$ consisting of a copy~$\ss\inv$ for each letter~$\ss$ of~$\SSS$. Words in such an alphabet~$\SSS \cup \SSS\inv$ will be called \emph{signed} words, and they will be denoted using bold characters, like~$\sw, \sv$, ... By contrast, $\ww, \vv$, ... will always refer to words in the alphabet~$\SSS$.

In order to encode the successive steps of a reversing process, we list the labels in the righmost paths that connect the South-West corner to the North-East corner in a reversing diagram. We decide that the contribution of an $\ss$-labeled edge in such a path is the letter~$\ss$ if the edge is crossed according to its orientation, and~$\ss\inv$ in the opposite case. For instance, the encoding of the (unique) SW-to-NE path in the initial diagram containing~$\ww$ (vertical) and $\www$ (horizontal) is $\ww\inv \www$, where $\ww\inv$ is ``$\ww$ read in the wrong direction'', \ie, is the word obtained from~$\ww$ by replacing each letter with its inverse and reversing the order of letters.

With such coding conventions, performing one step of the reversing method, \ie, closing some open pattern 
\VR(10,5)\begin{picture}(15,0)(-3,5)
\put(0,0){\includegraphics{Open3.eps}}
\put(-2,3){$\ss$}
\put(7,12){$\sss$}
\end{picture}
into
\VR(6,6)\begin{picture}(18,0)(-3,5)
\put(0,0){\includegraphics{Close3.eps}}
\put(-2,3){$\ss$}
\put(7,12){$\sss$}
\put(12,5){$\vv$}
\put(5,-2.5){$\vvv$}
\end{picture}
corresponds to replacing a subword~$\ss\inv \sss$ with a word~$\vvv \vv\inv$ such that $\ss\vvv = \sss\vv$ is a relation of~$\RR$. This includes the case of 
\VR(6,6)\begin{picture}(15,0)(-3,5)
\put(0,0){\includegraphics{Close4.eps}}
\put(-2,3){$\ss$}
\put(7,12){$\ss$}
\end{picture}, which corresponds to deleting a subword~$\ss\inv \ss$, \ie, using $\ew$ for the empty word, replacing it with~$\ew$, which is also~$\ew \ew\inv$, hence the same basic step provided  $\ss = \ss$ is considered to implicitly belong to~$\RR$.

\begin{defi}[\bf reversing]
\label{D:Reversing}
For $(\SSS, \RR)$ a semigroup presentation and $\sw, \sww$ signed words in the alphabet~$\SSS \cup \SSS\inv$, we say that $\sw$ \emph{reverses to~$\sww$ (with respect to~$\RR$) in one step}, denoted $\sw \revR^1 \sww$, if there exist a relation $\ss \vvv = \sss \vv$ of~$\RR$ and signed words~$\su, \suu$ satisfying
\begin{equation}
\label{E:Reversing}
\sw = \su \, \ss\inv \sss \, \suu
\qquad\mbox{and}\qquad
\sww= \su \, \vvv \vv\inv \, \suu.
\end{equation}
We say that $\sw$ \emph{reverses} to~$\sww$ in $\kk$~steps, denoted $\sw \revR^\kk \sww$, if there exist words $\sw_0, ..., \sw_\kk$ satisfying $\sw_0 = \sw$, $\sw_\kk = \sww$ and $\sw_\ii \revR^1 \sw_{\ii+1}$ for each~$\ii$. In this case, the sequence $(\sw_0, ..., \sw_\kk)$ is called an \emph{$\RR$-reversing sequence} from~$\sw$ to~$\sww$. We write $\sw \revR \sww$, or simply $\sw \rev \sww$, if $\sw \revR^\kk \sww$ holds for some~$\kk$, \ie, if there exists at least one $\RR$-reversing sequence connecting~$\sw$ to~$\sww$.
\end{defi}

If we call the letters of~$\SSS$ positive, and those of~$\SSS\inv$ negative, then \eqref{E:Reversing} shows that, in terms of the encoding words, reversing amounts to replacing a negative--positive subword with a positive--negative word. This is the origin of the terminology. Of course, except in the case of a commutation relation~$\ss \sss = \sss \ss$, reversing the subword~$\ss\inv \sss$ does not readily means keeping the letters and changing their order.\footnote{by the way, the names ``redressing'' or ``rectifying'' might have been more appropriate}

\begin{exam}
The successive SW-to-NE paths in the diagrams of Figure~\ref{F:RevStrat} correspond to the reversing sequence
\begin{multline*}
\ttB \ttB \ttB \ttD \, \framebox{$\ttC \tta$} \,\ttc \tta \tta \tta
\rev^1 \ttB \ttB \ttB \,\framebox{$\ttD \tta$}\, \ttB \ttc \tta \tta \tta
\rev^1 \ttB \ttB \,\framebox{$\ttB \ttb$}\, \ttD \ttB \ttc \tta \tta \tta\\
\rev^1 \ttB \ttB \ttD \,\framebox{$\ttB \ttc$}\, \tta \tta \tta
\rev^1 \ttB \ttB \ttD \ttc \,\framebox{$\ttA \tta$}\, \tta \tta
\rev^1 \ttB \ttB \,\framebox{$\ttD \ttc$}\, \tta \tta,
\end{multline*}
in which we used $\ttA, \ttB, ...$ for~$\tta\inv, \ttb\inv, ...$ and we framed the length-two subword that is to be reversed at each step.
\end{exam}

The words that are terminal with respect to $\RR$-reversing are those that contain no length-two subword of the form~$\ss\inv \sss$ such that $\RR$ contains a relation $\ss ... = \sss ...$\,. Among such words are all the words of the form  $\vvv \vv\inv$ where $\vv$ and~$\vvv$ are words in the alphabet~$\SSS$, since such words contain no subword~$\ss\inv \sss$ at all.

A reversing diagram starting with~$\ww\inv \www$ and finishing with~$\vvv \vv\inv$, where $\ww, \www, \vv, \vvv$ are positive words, projects to a van Kampen diagram for~$\ww \vvv$ and~$\www \vv$, so the following is straightforward:

\begin{lemm}
\label{L:Equiv}
For $\ww, \www, \vv, \vvv$ in~$\SSS^*$, the relation $\ww\inv \www \revR\vvv \vv\inv$,  \ie, the existence of an $\RR$-reversing diagram
\VR(6,7)\begin{picture}(23,10)(-3,5)
\put(0.5,0){\includegraphics{CommonMultiple.eps}}
\put(-2,5){$\ww$}
\put(7,12){$\www$}
\put(17,5){$\vv$}
\put(7,-2){$\vvv$}
\put(7,5){$\revR$}
\end{picture}, implies $\ww \vvv \eqpR \www \vv$.
\end{lemm}

\begin{rema}
In Definition~\ref{D:Reversing}, we consider length-two subwords~$\ss\inv \sss$ only. One can modify the definition so as to allow longer negative--positive subwords~$\uu\inv \uuu$ where $\uu$ and $\uuu$ are nonempty words of~$\SSS^*$, and declare that $\uu\inv \uuu$ reverses to~$\vvv \vv\inv$ whenever $\uu \vvv = \uuu \vv$ is a relation of~$\RR$. This new notion of reversing is actually equivalent to the previous one. Indeed, if we use $\revR^*$ for the new notion, it is clear that $\revR^*$ includes~$\revR$. Conversely, to see that $\ww \revR^* \www$ implies $\ww \revR \www$, it is sufficient to consider the elementary step $\uu\inv \uuu \revR^* \vvv \vv\inv$. Write $\uu = \ss_1 \, ... \,  \ss_\pp$ and $\uuu = \sss_1 \, ... \,  \sss_\qq$. By hypothesis, $\ss_1 \, ... \,  \ss_\pp \vvv = \sss_1 \, ... \,  \sss_\qq \vv$ is a relation of~$\RR$, and we find
\begin{multline*}
\uu\inv \uuu = \ss_\pp\inv \, ... \,  \ss_1\inv \sss_1 \, ... \,  \sss_\qq \revR \\
\ss_\pp\inv \, ... \,  \ss_2\inv \ss_2 \, ... \,  \ss_\pp \vvv \vv\inv \sss_\qq{}\inv \, ... \,  \sss_2{}\inv \sss_2 \, ... \,  \sss_\qq \revR^{\pp + \qq -2} \vvv \vv\inv,
\end{multline*}
hence $\uu\inv \uuu \revR \vvv \vv\inv$. The only difference between~$\revR$ and~$\revR^*$ lies in the number of reversing steps, but not in the words that can be reached. In terms of diagrams, this means that we can freely gather reversing tiles so as to avoid trivial steps of the form~$\ss\inv \ss \rev \ew$. For instance, in the context of Figure~\ref{F:RevStrat}, at Step~4, instead of reversing $\ttB \ttc \tta$ into $\ttc \ttA \tta$, and then identifying the two $\tta$-edges to obtain~Ê$\ttc$, we can directly consider the negative--positive pattern $\ttB (\ttc \tta)$, and reverse it to~$\ttc$ since $\ttb \cdot \ttc = (\ttc \tta) \cdot \ew$ is a relation of the presentation. This amounts to gathering the two tiles
\VR(9,8)\begin{picture}(25,0)(-3,5.5)
\put(-0.5,0){\includegraphics{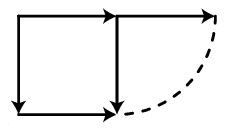}}
\put(4,12){$\ttc$}
\put(14,12){$\tta$}
\put(-2,5.5){$\ttb$}
\put(11,5.5){$\tta$}
\put(4,-1){$\ttc$}
\end{picture}
into the unique tile
\VR(9,8)\begin{picture}(25,0)(-3,5.5)
\put(-0.5,0){\includegraphics{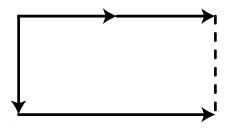}}
\put(4,12){$\ttc$}
\put(14,12){$\tta$}
\put(-2,5.5){$\ttb$}
\put(9,-1){$\ttc$}
\end{picture}.
\end{rema}

\subsection{Left-reversing}
\label{S:LeftReversing}

By construction, subword reversing refers to a preferred direction, namely tessellating van Kampen diagrams starting from the source vertex and proceeding toward the target vertex, \ie, equivalently, reversing negative--positive subwords~$\ss\inv \sss$ into positive--negative words~$\vvv \vv\inv$. A symmetric approach is possible, namely starting from the target vertex of the van Kampen diagram and trying to build a tessellation from right to left. In syntactic terms, this amounts to reversing a positive--negative subword~$\sss \ss\inv$ into a negative--positive word~$\vv\inv \vvv$ such that $\vv\sss = \vvv\ss$ is a relation of the considered presentation. Hereafter, the process considered in the previous sections will be called \emph{right-reversing} (assuming that the source is drawn on the left of the diagram, the process goes to the right; also it provides common right-multiples), whereas the symmetric version where one starts from the target vertex will be called \emph{left-reversing} and denoted~$\revlR$, or simply~$\revl$. Of course, the properties of right- and left-reversings are symmetric, and it is enough to concentrate on one side, except when one combines both procedures as in Sections~\ref{S:DoubleReversing} and~\ref{S:Mixed} below.

\begin{rema}
\label{R:LeftRev}
Left-reversing is \emph{not} the inverse of right-reversing: $\ww \rev \nobreak\www$ does not necessarily imply $\www \revl \ww$. For instance, for each~$\ss$ in the alphabet, $\ss\inv \ss \rev \ew$ holds, but $\ew \revl \ss\inv \ss$ fails: both for right- and for left-reversing, the empty word is reversible to no word other than itself. One could recover symmetry by changing the definition and deciding that $\ss\inv \ss$ right-reverses to~$\ss\inv \ss$, and that $\ss \ss\inv$ left-reverses to~$\ss\inv \ss$. This approach seems definitely poor: by doing so, even in favorable cases like the one of braid monoids---see Example~\ref{X:Braid}---one loses termination. For instance, starting from $\ww = \siginv1 \siginv2 \sig1 \sig2$, right-reversing~$\ww$ in this modified way would lead to the infinite series $(\sig1 \sig2 \sig 1 \sig2 \sig1 \sig2)^\kk \ww (\sig1 \sig2 \sig 1 \sig2 \sig1 \sig2)^{-\kk}$, instead of to~$\sig2 \siginv1$ obtained with the initial process, see Figure~\ref{F:Symmetric}.
\end{rema}

\begin{figure}[htb]
\begin{picture}(80,35)(0,-2)
\put(-0.5,-0.5){\includegraphics{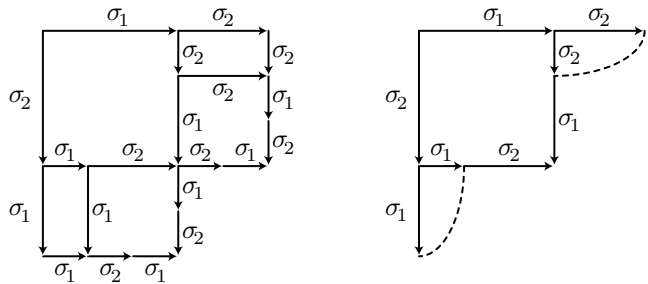}}
\put(-4,6){$\sig1$}
\put(-4,21){$\sig2$}
\put(7,6){$\sig1$}
\put(19,27){$\sig2$}
\put(19,18){$\sig1$}
\put(19,9){$\sig1$}
\put(19,3){$\sig2$}
\put(31,27){$\sig2$}
\put(31,21){$\sig1$}
\put(31,15){$\sig2$}
\put(2,-2){$\sig1$}
\put(8,-2){$\sig2$}
\put(14,-2){$\sig1$}
\put(2,14){$\sig1$}
\put(11,14){$\sig2$}
\put(20,14){$\sig2$}
\put(26,14){$\sig1$}
\put(9,32){$\sig1$}
\put(23,32){$\sig2$}
\put(23,22){$\sig2$}

\put(46,6){$\sig1$}
\put(46,21){$\sig2$}
\put(59,32){$\sig1$}
\put(73,32){$\sig2$}
\put(52,14){$\sig1$}
\put(61,14){$\sig2$}
\put(69,27){$\sig2$}
\put(69,18){$\sig1$}
\end{picture}
\caption{\sf Modifying the definition of right-reversing so as to make it symmetric (left diagram) may lead to a non-terminating process; compare with usual reversing (right diagram).}
\label{F:Symmetric}
\end{figure}

\section{Subword reversing: range}
\label{S:Range}

At this point, we have introduced a tentative strategy for constructing van Kampen diagrams or, equivalently, for sorting the negative and the positive letters of a signed word in a symmetrized alphabet~$\SSS \cup \SSS\inv$. However, we already observed that this strategy need not work in every case, and the interest of our approach is not clear yet. The good news is that there exist nontrivial cases in which the reversing strategy works and, better, there exist practical criteria for identifying such favorable cases and even forcing initially bad presentations to become good at the expense of adding some redundant relations.

\subsection{Complete presentations}
\label{S:Complete}

Let $(\SSS, \RR)$ be a semigroup presentation, and $\ww, \www$ be words in the alphabet~$\SSS$. By construction (or by Lemma~\ref{L:Equiv}), if $\ww\inv \www \revR \ew$ holds, \ie, if everything vanishes when $\RR$-reversing is applied to~$\ww\inv \www$---we recall that $\ew$ denotes the empty word---then we have $\ww \eqpR \www$. So, in such a case, the reversing strategy is successful and provides an $\RR$-derivation from~$\ww$ to~$\www$. The good case is when the previous implication is an equivalence, \ie, when reversing always detects equivalence. 

\begin{defi}[\bf complete]
\label{D:Complete}
 A semigroup presentation $(\SSS, \RR)$ is called \emph{complete (with respect to right-reversing)} if, for all words~$\ww, \www$ in the alphabet~$\SSS$,
\begin{equation}
\label{E:Complete}
\ww \eqpR \www
\mbox{\qquad implies \qquad}
\ww\inv \www \revR \ew. 
\end{equation}
\end{defi}

As recalled above, the converse of~\eqref{E:Complete} is always true, so, if $(\SSS, \RR)$ is complete, \eqref{E:Complete} is actually an equivalence.

\begin{exam}
Our favourite presentation, namely that of Example~\ref{X:Preferred}, is certainly not complete: we have seen that the words~$\tta\ttc\tta\tta\tta$ and $\ttc\ttd\ttb\ttb\ttb$ are equivalent, but that the reversing strategy fails to find a van Kampen diagram for this pair of words: $(\tta\ttc\tta\tta\tta)\inv (\ttc\ttd\ttb\ttb\ttb)$ does not reverse to the empty word. So \eqref{E:Complete} fails for these words.
\end{exam}

\begin{rema}
\label{R:WordProblem}
If we start with a finite presentation~$(\SSS, \RR)$ and every $\RR$-reversing sequence is finite, then completeness implies the solvability of the word problem for~$\Mon\SSS\RR$. Indeed, $\ww \eqpR \www$ is then equivalent to $\ww\inv \www \revR \ew$, and the latter can be decided by exhaustively constructing all $\RR$-reversing sequences from~$\ww\inv \www$ and checking whether the empty word occurs. But, if there exist infinite $\RR$-reversing sequences, we may be unable to decide whether $\ww\inv \www$  reverses to the empty word, and, therefore, to possibly prove that two words~$\ww, \www$ are not $\RR$-equivalent. In this case, even if the presentation is complete, we need not obtain a solution to the associated word problem: completeness and solvability of the word problem are different questions---see Section~\ref{S:WordProblem} below.
\end{rema} 

It is easy to see that complete presentations always exist: for every monoid~$\MM$ with no nontrivial invertible element, the full presentation consisting of a generator~$\underline\xx$ for each element~$\xx$ of~$\MM$ and a list of all relations $\underline\xx \, \underline\yy = \underline\zz$ for $\xx, \yy, \zz$ satisfying $\xx \yy = \zz$ in~$\MM$ is complete. But this result is useless: in most cases, our aim is to investigate a (not yet known) monoid starting from a presentation, in particular to solve the word problem, whereas writing the above full presentation would require a prior solution to that word problem. Moreover, we shall mainly be interested in complete presentations that are as small (finite) as possible, which is rarely the case for a full presentation. 

In this context, the three natural problems are: 

\begin{ques}
\label{Q:What}
$(i)$ How to recognize completeness?

$(ii)$ What to do with a non-complete presentation? 

$(iii)$ What to do with a complete presentation? 
\end{ques}

As can be expected, the answer to Question~\ref{Q:What}$(ii)$ will be: try to make the presentation complete, whereas the answer to Question~\ref{Q:What}$(iii)$ will be: prove properties of the monoid.

Before addressing these questions, we mention an alternative definition of completeness.

\begin{lemm}
\label{L:Complete}
A semigroup presentation~$(\SSS, \RR)$ is complete if and only if, for all~$\uu, \vv, \uuu, \vvv$ in~$\SSS^*$ satisfying $\uu \vvv \eqpR \vv \uuu$, there exist~$\uuuu, \vvvv, \ww$ in~$\SSS^*$ satisfying $\uu\inv \vv \revR \vvvv\uuuu{}\inv$, $\uuu \eqpR\nobreak \uuuu \ww$, and $\vvv \eqpR \vvvv \ww$.
\end{lemm}

Roughly speaking, completeness holds if every common right-multiple relation factors through a reversing. The equivalence with Definition~\ref{D:Complete} is easily established.

\subsection{The cube condition}

As for Question~\ref{Q:What}$(i)$, there exists a satisfactory answer, or, actually, several satisfactory answers covering various cases. The key notion is as follows.

\begin{defi}[\bf cube condition]
\label{D:Cube}
Assume that $(\SSS, \RR)$ is a semigroup presentation, and $\uu, \uuu, \uuuu$ are words in the alphabet~$\SSS$. We say that $(\SSS, \RR)$ satisfies the \emph{cube condition for~$(\uu, \uuu, \uuuu)$} if
\begin{equation}
\label{E:Cube}
\uu\inv \uuuu \uuuu{}\inv \uuu \revR \vvv \vv\inv
\mbox{\qquad implies \qquad}
(\uu \vvv)\inv (\vv \uuu) \revR \ew.
\end{equation}
For~$\XX$ included in~$\SSS^*$, we say that $(\SSS, \RR)$ satisfies the cube condition \emph{on~$\XX$} if it satisfies the cube condition for every triple~$(\uu, \uuu, \uuuu)$ with $\uu, \uuu, \uuuu$ in~$\XX$.
\end{defi}

Pictorially, \eqref{E:Cube} means that
\VR(11,10)\begin{picture}(29,0)(-3,7)
\put(0,0){\includegraphics{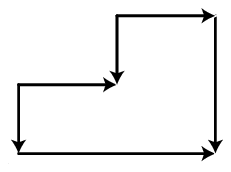}}
\put(13,5){$\rev$}
\put(-2,4){$\uu$}
\put(7,11){$\uuuu$}
\put(3,9){$\uuuu$}
\put(15,16){$\uuu$}
\put(10,-2){$\vvv$}
\put(21.5,7){$\vv$}
\end{picture}
implies
\VR(10,10)\begin{picture}(24,0)(-3,7)
\put(0,0){\includegraphics{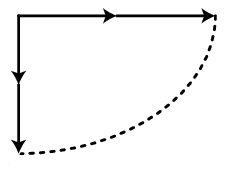}}
\put(8,7){$\rev$}
\put(-2,11){$\uu$}
\put(-2,4){$\vvv$}
\put(4,16){$\uuu$}
\put(14,16){$\vv$}
\end{picture}.

We insist that what \eqref{E:Cube} says is that, for each reversing sequence starting from $\uu\inv \uuuu \uuuu{}\inv \uuu$ \emph{and} terminating in a positive--negative word~$\vvv \vv\inv$, we have $(\uu \vvv)\inv (\vv \uuu) \revR \ew$. So, in particular, if there is no  such sequence---for instance because every sequence is infinite, or because the sequences get stuck by lack of a relation---then the cube condition is vacuously true.

The cube condition is called so because it expresses that, if we start with three edges labeled~$\uu, \uuu, \uuuu$ in a three-dimensional space and construct three reversing diagrams, respectively from~$(\uu, \uuuu)$, $(\uuuu, \uu)$, and from the two edges extending~$\uuuu$---thus making three faces of a cube---then there is a way to complete that cube with reversing diagrams as shown in Figure~\ref{F:Cube}.

\begin{figure}[htb]
$$\begin{picture}(40,40)(0,2)
\put(0,0){\includegraphics{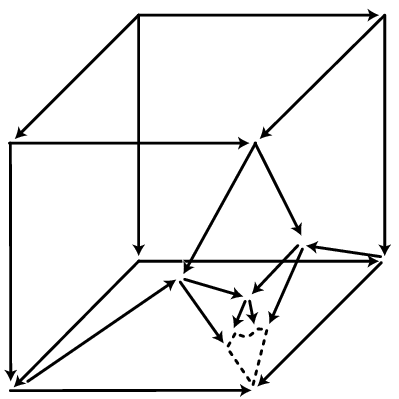}}
\put(6,34){$\uu$}
\put(15,31){$\uuuu$}
\put(26,40){$\uuu$}
\end{picture}$$
\caption{\sf The cube condition: whichever the way of drawing three faces of a cube using reversing diagrams, one can complete the cube using reversing diagrams as shown.}
\label{F:Cube}
\end{figure}

The following is easy.

\begin{prop}
\label{P:Cube}
A semigroup presentation $(\SSS, \RR)$ is complete with respect to right-reversing if and only if it satisfies the cube condition on~$\SSS^*$.
\end{prop}

\begin{proof}[Proof (sketch)]
Assume $\uu\inv \uuuu \uuuu{}\inv \uuu \revR \vvv \vv\inv$. Using Lemma~\ref{L:Equiv}, one easily sees that the words~$\uu \vvv$ and~$\vv \uuu$ are $\RR$-equivalent. So, if $(\SSS, \RR)$ is complete, we must have $(\uu \vvv)\inv (\uuu \vv) \revR \ew$, and the cube condition is satisfied for~$(\uu, \uuu, \uuuu)$.

Conversely, consider the binary relation $\uu\inv \uuu \revR \ew$ on~$\SSS^*$. This relation contains all relations of~$\RR$ and, by Lemma~\ref{L:Equiv}, it is included in~$\eqpR$. Now, by definition, $\eqpR$ is the smallest congruence that contains all relations of~$\RR$. So, in order to prove that $\uu\inv \uuu \revR \ew$ coincides with~$\eqpR$, it suffices to prove that $\uu\inv \uuu \revR \ew$ is itself a congruence. All required properties are easy, except transitivity. Now assume  $\uu\inv \uuuu \revR \ew$ and $\uuuu{}\inv \uuu \revR \nobreak \ew$. Then we have  $\uu\inv \uuuu \uuuu{}\inv \uuu \revR \ew$, so, if \eqref{E:Cube} holds for the triple $(\uu, \uuu, \uuuu)$, we deduce $\uu\inv \uuu \revR \ew$, the desired transitivity.
\end{proof}

What was introduced in Definition~\ref{D:Cube} is the \emph{right} cube condition, which is relevant for right-reversing. Of course, a symmetric \emph{left cube condition} is relevant for left-reversing.

\subsection{Homogeneous presentations}

As it stands, the completeness criterion of Proposition~\ref{P:Cube} is useless, as checking the cube condition for all triples of words is not feasible. Fortunately a much more tractable criterion is available whenever a mild additional hypothesis is satisfied, namely a noetherianity condition that prevents a given word to be equivalent to words of unbounded lengths.

\begin{defi}[\bf homogeneous]
\label{D:Homogeneous}
A semigroup presentation $(\SSS, \RR)$ is said to be \emph{(left)-homogeneous} if there exists an $\eqpR$-invariant mapping $\lb$ of~$\SSS^*$ to ordinals
satisfying, for every letter~$\ss$ in~$\SSS$ and every word~$\ww$ in~$\SSS^*$, \begin{equation}
\label{E:Homogeneous}
\lb(\ss \ww) > \lb(\ww).
\end{equation}
\end{defi}

The mapping~$\lb$ should be seen as a (weak) length function on the monoid~$\Mon\SSS\RR$. In usual cases, it can be assumed to take values in natural numbers, which are particular ordinals. A typical case is when all relations in~$\RR$ preserve the length of words, \ie, they have the form $\vvv = \vv$ where $\vvv$ and $\vv$ have the same length: then the length function is $\eqpR$-invariant and \eqref{E:Homogeneous} is trivially satisfied. Saying that a presentation~$(\SSS, \RR)$ is homogeneous (with a witness-function with values in~$\NNNN$) is equivalent to saying that the monoid~$\Mon\SSS\RR$ satisfies the condition~$(ii)$ of Theorems~1, 2, and~3 of the introduction.

The main result is as follows.

\begin{prop} \cite{Dgp}
\label{P:Cube}
Assume that $(\SSS, \RR)$ is a homogeneous semigroup presentation.  Then $(\SSS, \RR)$ is complete if and only if it satisfies the cube condition on~$\SSS$.
\end{prop}

The proof is a rather delicate induction involving the function~$\lb$ provided by the homogeneity assumption. The benefit with respect to Proposition~\ref{P:Cube} is clear: with the criterion of Proposition~\ref{P:Cube}, it is enough to check the cube condition for triples of letters. In particular, in the case of a finite presentation, only finitely many triples have to be considered.

\begin{exam}
\label{X:Main3}
The presentation of Example~\ref{X:Preferred} is homogeneous. Indeed, all relations are of the form $\vv = \vvv$ where $\vv$ and $\vvv$ have length two, so the length function can be used as the desired function~$\lb$. It is easy to check that the cube condition is satisfied for most triples of letters. For instance,  consider the triple $(\tta, \ttb, \ttc)$. Then $\ttA \ttb \ttB \ttc$ reverses to $\ttb \ttA$ and to $\ttd \ttb \ttA \ttA$---we recall that, in examples, we use $\ttA$ for $\tta\inv$, etc. So checking the cube condition for~$(\tta, \ttb, \ttc)$ means checking that both $(\tta \cdot \ttb)\inv (\ttc \cdot \tta)$ and $(\tta \cdot \ttd \ttb)\inv (\ttc \cdot \tta\tta)$, \ie, $\ttB \ttA \ttc \tta$ and $\ttB \ttD \ttA \ttc \tta \tta$, reverse to the empty word, which is the case indeed. On the other hand, we know that the presentation is not complete, so the cube condition must fail for some triple of letters. Actually it fails for~$(\ttc, \ttd, \tta)$. Indeed, $\ttC \tta \ttA \ttd$ reverses to~$\tta \tta \ttB \ttB$. Now $\ttA \ttA \ttC \ttd \ttb \ttb$ does not reverse to the empty word: as there is no relation $\ttc ... = \ttd ... $ in the presentation, this word reverses to no word but itself and, therefore, the cube condition fails for $(\ttc, \ttd, \tta)$.
\end{exam}

Many important families of semigroup presentations turn out to be complete with respect to subword reversing. This is the case in particular for all Artin--Tits presentations, which are those presentations in which all relations take the form
\begin{equation}
\label{E:Artin}
\ss \sss \ss \sss ... = \sss \ss \sss \ss ...
\end{equation}
with both sides of the same length. Artin's presentation of the braid groups, which involve such relations with words of length~$2$ and~$3$, are typical examples (see Example~\ref{X:Braid}). It is an interesting exercise to check the cube condition for a triple of letters pairwise
connected by relations of the type~\eqref{E:Artin}.

\subsection{Completion}
\label{S:Completion}

When we start with a semigroup presentation~$(\SSS, \RR)$ that turns out to be complete, then we are in the optimal case and the monoid~$\Mon\SSS\RR$ is directly eligible for the results explained in Section~\ref{S:Uses} below.

However, even if the initial presentation~$(\SSS, \RR)$ is not complete, typically, if the cube condition turns out to fail for some triple of letters, as in the case of Example~\ref{X:Main3}, then using subword reversing is not completely impossible. Indeed, assume that the cube condition fails for some triple $(\ss, \sss, \ssss)$. This means that we found words~$\vv, \vvv$ such that $\ss\inv \ssss \ssss{}\inv \sss$ reverses to~$\vvv \vv\inv$, but $(\ss \vvv)\inv (\sss \vv)$ does not reverse to the empty word. Now, in this case, we have $\ss \vvv \eqpR \sss \vv$, and $\ss\vvv$ and $\sss\vv$ are $\RR$-equivalent words whose equivalence is not detected by reversing. Let $\RRh$ be obtained by adding the relation $\ss \vvv = \sss \vv$ to~$\RR$. As $\ss\vvv \eqpR \sss\vv$ holds, the new relation is redundant, and the monoid~$\Mon\SSS{\RRh}$ coincides with~$\Mon\SSS\RR$. On the other hand, by construction, the word $(\ss\vvv)\inv (\ss\vvv)$ is $\RRh$-reversible to the empty word, as shows the $\RRh$-reversing sequence
$$\vvv{}\inv \ss\inv \sss \vv \rev \vvv{}\inv \vvv \vv\inv \vv \rev \vvv{}\inv \vvv \rev \ew.$$
In this way, we obtained a new presentation~$(\SSS, \RRh)$ of the same monoid, and we can check the cube condition for it. Notice that, as a new relation has been added, new possibilities of reversing have been added and the previously checked cases must be revisited.

In this way we obtain an iterative completion procedure that consists in adding redundant relations to the presentation. Two cases are possible. The good case is when one obtains a complete presentation after finitely many completion steps, in which case subword reversing is useful for investigating the considered monoid. The bad case is when the completion never comes to an end and leads to larger and larger presentations, in which case subword reversing is likely to be of no use. 

\begin{exam}
\label{X:Main4}
Returning to the case of Example~\ref{X:Main3}, we saw that the cube condition fails for $(\ttc, \ttd, \tta)$, as it leads to the equivalence $\ttc \tta \tta \eqp \ttd \ttb \ttb$, which is not detected by reversing. According to the general principle, adding the relation $\ttc \tta \tta = \ttd \ttb \ttb$ provides a new presentation of the same monoid. By construction, the new presentation is homogeneous (with the same witnessing pseudolength function), and we check again the cube condition for all triples of letters. For instance, with respect to the initial presentation, the word $\ttA \ttc \ttC \ttd$ reverses to no word of the form~$\vvv \vv\inv$ with $\vv, \vvv$ positive, so the cube condition for $(\tta, \ttd, \ttc)$ is vacuously true. With the completed presentation, $\ttA \ttc \ttC \ttd$ reverses to $\ttb \tta \ttB \ttB$, and we can check that $\ttA \ttB \ttA \ttd \ttb \ttb$ reverses to the empty word, so the cube condition is satisfied for that triple. It turns out that it is satisfied for all triple of letters and, therefore, the completed presentation
$$(\tta, \ttb, \ttc, \ttd \mid \tta\ttb = \ttb\ttc = \ttc\tta,
\ttb\tta = \ttd\ttb = \tta\ttd, \ttc \tta \tta = \ttd \ttb \ttb)$$
is complete. When we revisit with this extended presentation the equivalent words of Example~\ref{X:Preferred}, then, as expected, we are no longer stuck after five reversing steps and we finally obtain a van Kampen diagram witnessing the equivalence of the initial words, as expected (see Figure~\ref{F:ReversingBis}).
\end{exam}

\begin{figure}[htb]
$$\begin{picture}(52,33)(0,2)
\put(1.5,0){\includegraphics{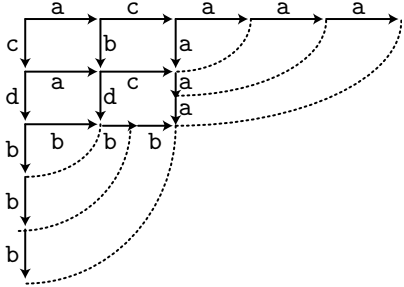}}
\put(0,3){$\ttb$}
\put(0,10){$\ttb$}
\put(0,17){$\ttb$}
\put(0,24){$\ttd$}
\put(0,31){$\ttc$}
\put(6,36){$\tta$}
\put(16,36){$\ttc$}
\put(26,36){$\tta$}
\put(36,36){$\tta$}
\put(46,36){$\tta$}
\put(13,31){$\ttb$}
\put(23,31){$\tta$}
\put(6,26){$\tta$}
\put(16,26){$\ttc$}
\put(13,24){$\ttd$}
\put(23,26){$\tta$}
\put(23,22.5){$\tta$}
\put(6,18){$\ttb$}
\put(13,18){$\ttb$}
\put(19,18){$\ttb$}
\end{picture}$$
\caption{\sf After adding the (redundant) relation $\ttc \tta \tta = \ttd \ttb \ttb$ to the presentation, we can extend the reversing diagram of Figure~\ref{F:Reversing}. We finish with dotted arcs everywhere, thus obtaining a van Kampen diagram witnessing that the initial words are equivalent.}
\label{F:ReversingBis}
\end{figure}

\begin{rema}
\label{R:Grobner}
Several other algorithmic methods consist of investigating a presented monoid by
iteratively adding redundant relations so as to satisfy certain completeness conditions, and it is natural to look for possible connections with the current completion. A typical case is that of the Gr\"obner--Shirshov bases. Apart from some isolated examples, it turns out that the reversing completion and the Gr\"obner--Shirshov completion are unrelated in general (the reversing completion being much smaller in most cases)~\cite{Aut}. As for the Knuth--Bendix completion, it is defined in a framework of oriented rewrite rules, so a comparison does not really make sense. 
\end{rema}

\subsection{The case of complemented presentations}

So far, we made no restriction about the number of relations in the considered presentations. When we add such restrictions, things may become more simple and new completeness criteria appear.

\begin{defi}[\bf complemented]
\label{D:Complemented}
A semigroup presentation $(\SSS, \RR)$ is called \emph{(right)-complemented} if, for each~$\ss$ in~$\SSS$, there is no relation~$\ss ... = \ss ... $ in~$\RR$ and,  for  $\ss, \sss$ distinct in~$\SSS$, there is at most one relation $\ss \, ... = \sss \, ... $ in~$\RR$.
\end{defi}

For instance, the presentation of Example~\ref{X:Preferred} is not complemented: it contains two relations of the form $\tta ...= \ttb ...$, namely $\tta\ttb = \ttb \ttc$ and $\tta \ttd = \ttb \tta$. Note that the completion process of Section~\ref{S:Completion} can delete the possible complemented character of the initial presentation: for instance, starting with the complemented presentation $(\tta, \ttb, \ttc \mid \tta\ttc = \ttc\tta, \ttb\ttc = \ttc\ttb, \tta\ttb = \ttb\tta\ttc)$ of the Heisenberg monoid, the completion process leads to adding the relation $\tta\ttb = \ttc\ttb\tta$, thus yielding a non-complemeneted presentation with two relations $\tta... = \ttb...$\, \cite{Dgp}.

For a complemented presentation, reversing is a deterministic process: at each step, at most one relation is eligible, and, therefore, for every initial signed word~$\sw$, only one reversing diagram can be constructed from~$\sw$---but several reversing sequences may start from~$\sw$ as there may be several ways of enumerating the tiles of the reversing diagram.

In such a framework, a binary operation on (positive) words is naturally associated with reversing.

\begin{defi}[\bf complement]
\label{D:Complement}
For $(\SSS, \RR)$ a complemented semigroup presentation and $\ww, \www$ in~$\SSS^*$, the \emph{$\RR$-complement} of~$\www$ in~$\ww$, denoted $\ww\cpR \www$ or, simply, $\ww \cp \www$, (``$\ww$ under $\www$''), is the unique word~$\vvv$ of~$\SSS^*$ such that $\ww\inv \www$ reverses to~$\vvv \vv\inv$ for some~$\vv$ in~$\SSS^*$, if such a word exists.
\end{defi}

The symmetry of reversing guarantees that, if $\ww \cp \www$ exists, then so does $\www \cp \ww$ and, in this case, $\ww\inv \www$ reverses to $(\ww \cp \www) (\www \cp \ww)\inv$. So $\ww \cp \www$ and $\www \cp \ww$ are the last two sides of the reversing rectangle built on~$\ww$ and~$\www$, when the latter exists, \ie, when the (unique) maximal $\RR$-reversing sequence from~$\ww\inv \www$ is finite:
$$\VR(14,3)\begin{picture}(23,13)
\put(0,0){\includegraphics{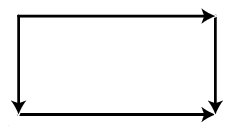}}
\put(-2.5,5){$\ww$}
\put(9.5,12){$\www$}
\put(9,5){$\revR$}
\put(22,5){$\www\cpR\ww$ .}
\put(7,-2.5){$\ww\cpR\www$}
\end{picture}$$
It is then easy to restate the cube condition as an algebraic condition satisfied by the complement operation.

\begin{prop}\cite{Dgp}
\label{P:Restate}
Assume that $(\SSS, \RR)$ is a complemented semigroup presentation. Then, for all words~$\uu, \uuu, \uuuu$ in~$\SSS^*$, the following are equivalent:

$(i)$ $(\SSS, \RR)$ satisfies the cube condition on~$\{\uu, \uuu, \uuuu\}$;

$(ii)$ either $((\uu \cp \uuu) \cp (\uu \cp \uuuu)) \cp ((\uuu \cp \uu) \cp (\uuu \cp \uuuu)))$ is the empty word or it is undefined, and the same holds for all permutations of~$\uu, \uuu, \uuuu$.

$(iii)$ either $(\uu \cp \uuu) \cp (\uu \cp \uuuu)$ and $(\uuu \cp \uu) \cp (\uuu \cp \uuuu))$ are $\RR$-equivalent or they are not defined, and the same holds for all permutations of~$\uu, \uuu, \uuuu$.
\end{prop}

The equivalence of~$(i)$ and~$(ii)$ is an amusing application of the statement of the cube condition in a complemented framework, and it requires to simultaneously consider the triples $(\uu, \uuu, \uuuu)$, $(\uuu, \uuuu, \uu)$, and $(\uuuu, \uu, \uuu)$. The sufficiency of~$(iii)$ is slightly more delicate to establish. 

It may be noted that, in the complemented context, the cube of Figure \ref{F:Cube} takes the more simple---and more cube-like---form displayed in Figure~\ref{F:CubeBis}.

\begin{figure}[htb]
\begin{picture}(99,42)(0,2)
\put(0,0){\includegraphics{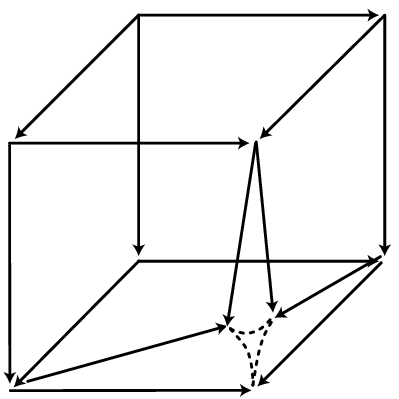}}
\put(6,34){$\uu$}
\put(15,31){$\uuuu$}
\put(26,40){$\uuu$}
\put(22,4){$\ew$}
\put(25,8){$\ew$}
\put(27.5,5){$\ew$}
\put(55,0){\includegraphics{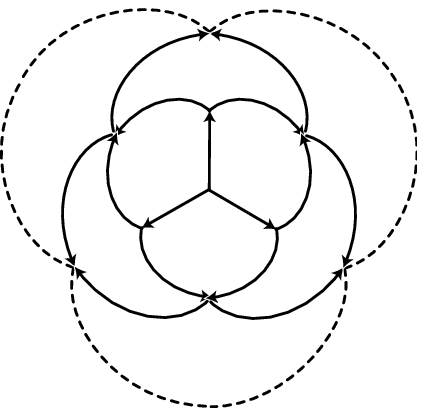}}
\put(73,26){$\uu$}
\put(80,21){$\uuu$}
\put(71.5,17){$\uuuu$}
\put(57,37){$\ew$}
\put(98,20){$\ew$}
\put(64,2){$\ew$}
\end{picture}
\caption{\sf The cube condition in a complemented context: when we draw the six faces of the cube, then reversing the three small triangular sectors leads to empty words everywhere, and the cube closes (left diagram); equivalently, starting from three edges, we use reversing to close three faces, and then repeat the process twice: at the end, everything vanishes (right diagram).}
\label{F:CubeBis}
\end{figure}

\begin{exam}
\label{X:ArtinCube}
As it is not complemented, our preferred example, namely the presentation of Example~\ref{X:Preferred}, is not eligible for the criterion of Proposition~\ref{P:Restate}. But all Artin--Tits presentations, which involve relations of the form~\eqref{E:Artin}, are eligible, and the criterion applies. For instance, it is an easy exercise to check that, for all values of the indices~$\ii, \jj, \kk$, the braid words $(\sig\ii \cp \sig\jj) \cp (\sig\ii \cp \sig\kk)$ and $(\sig\jj \cp \sig\ii) \cp (\sig\jj \cp \sig\kk)$ are equivalent (due to the symmetries of the braid relations, only three cases are to be considered, acoording to whether the indices are neighbors or not).
\end{exam}

For a presentation that is both homogeneous and complemented, the completeness criterion of Proposition~\ref{P:Cube} applies, and it can be restated using the equivalent forms of Proposition~\ref{P:Restate}. However, an alternative criterion also exists in the complemented case, which is valid even for a non-homogeneous presentation.

\begin{prop} \cite{Dgk}
\label{P:CubeBis}
Assume that $(\SSS, \RR)$ is a complemented semigroup presentation, and $\SSSh$ is a subset of~$\SSS^*$ that includes~$\SSS$ and is closed under complement, in the sense that, for all $\ww, \www$ in~$\SSSh$, the word~$\ww \cp \www$ lies in~$\SSSh$ whenever it exists.  Then $(\SSS, \RR)$ is complete if and only if it satisfies the cube condition on~$\SSSh$.
\end{prop}

Thus, in the complemented case, checking the cube condition not only on letters, but also on the closure of letters under complement enables one to forget about the homogeneity condition, which may be uneasy to establish (in terms of complexity hierarchies, this is a complete $\Pi_1^1$-condition, hence far from decidable).

\begin{rema}
We do not claim that the criteria of Propositions~\ref{P:Cube} and~\ref{P:CubeBis} are optimal, but it seems difficult to extend them much. In particular, all hypotheses are significant. For instance, $(\tta, \ttb, \ttc \mid \ \tta = \ttb^2 \ttc, \ttb \tta = \ttc, \ttc \tta = \ttc)$ is an example of a complemented presentation for which the cube condition holds for each triple of letters and there exists a finite set of words that is closed under complement, namely $\SSSh = \{\tta, \ttb, \ttc, \ew, \ttb\ttc\}$. Nevertheless the presentation is incomplete: $\tta$ and $\ttb \ttc \tta^2$ are equivalent words but  $\ttA \ttb \ttc \tta^2$ reverses to~$\tta^3$ and not to~$\ew$. This is compatible with the above criteria, since the presentation is not homogeneous ($\tta^3$ is equivalent to the empty word), and the cube condition fails for $(\tta, \ttb\ttc, \ttc)$, a triple from~$\SSSh$.
\end{rema}

We conclude this section with an alternative characterization of completeness (but one that leads to no practical criterion) involving the operation~$\cpR$ of Definition~\ref{D:Complement}.

\begin{prop}
\label{P:Complement}
A complemented presentation~$(\SSS, \RR)$ is complete if and only if $\cpR$ is compatible with~$\eqpR$, \ie, the conjunction of~$\uuu \eqpR \uu$ and~$\vvv \eqpR \vv$ implies $\uuu \cpR \vvv \eqpR \uu \cpR \vv$, this meaning that either the two expressions exist and are equivalent, or that neither exists.
\end{prop}

One implication is specially simple: by construction, $\ww \cpR \ww = \ew$ always holds, so, if $\cpR$ is compatible with~$\eqpR$, then $\www \eqpR \ww$ implies $\www \cpR \ww = \ww \cpR \www = \ew$, because $\ew$ is $\eqpR$-equivalent to no nonempty word, and this means that $\ww\inv \www$ reverses to~$\ew$.

No extension of Proposition~\ref{P:Complement} to the non-complemented case is known.

\section{Subword reversing: uses}
\label{S:Uses}

We now turn to the uses of subword reversing. So, here, we assume that we have a complete semigroup presentation~$(\SSS, \RR)$, and explain which properties of the monoid~$\Mon\SSS\RR$ or of the group~$\Gr\SSS\RR$ can be established. The general philosophy is that, when a presentation is complete, several properties that are difficult to prove in general become easy to read, the most important one being cancellativity. 

The successive topics addressed in this section are: proving cancellativity, proving the existence of least common multiples, solving the word problems, recognizing and working in Garside monoids, obtaining minimal fractionary decompositions, and, finally, proving embeddability of a monoid in a group.

\subsection{A cancellativity criterion}

Recognizing whether a presented monoid admits cancellation\footnote{\ie, whether $\xx\yy = \xx\zz$ or $\yy \xx = \zz \xx$ implies $\yy = \zz$ (respectively, left- and right-cancellativity)} is a difficult question. A well-known criterion of Adyan~\cite{Adj}, see also~\cite{Rem}, is often useful, but it is valid only for those presentations~$(\SSS, \RR)$ in which there is no cycle for the binary relation on~$\SSS$ that connects two letters~$\ss, \sss$ if there is a relation $\ss ... = \sss ... $ in~$\RR$. In particular, the criterion is not valid whenever there exists a pair of letters with at least two relations $\ss ... = \sss ... $, or a triple of letters with at least one relation $\ss ... = \sss ... $ for each pair. By contrast, whenever we have a complete presentation---hence in a context where there are often many relations---we have the following very simple criterion.

\begin{prop}\cite{Dgp}
\label{P:Cancellation}
Assume that $(\SSS, \RR)$ is a complete semigroup presentation. Then the monoid~$\Mon\SSS\RR$ is left-cancellative if and only if $\vv\inv \vvv \revR \ew$ holds for each relation of the form $\ss \vv = \ss \vvv$ in~$\RR$. In particular, a sufficient condition for $\Mon\SSS\RR$ to be left-cancellative is that there is no relation of the form $\ss \vv = \ss \vvv$ in~$\RR$.
\end{prop}

\begin{proof}[Proof (in the particular case when there is no relation $\ss \vv = \ss \vvv$ in~$\RR$)]
Assume $\ss \ww \eqpR \ss \www$. We want to prove $\ww \eqpR \www$. The completeness of~$(\SSS, \RR)$ implies $(\ss\ww)\inv (\ss\www) \revR \ew$, \ie, there exists a sequence 
$$\ww\inv \ss\inv \ss \www \rev^1 ... \rev^1 ... \rev^1 \ew.$$
Now the first step in the above reversing sequence must be $\ww\inv \ss\inv \ss \www \rev \ww\inv \www$, since there is no other possibility. But then the sequel of the sequence witnesses that $\ww\inv \www$ reverses to the empty word, hence implies that $\ww$ and $\www$ are $\RR$-equivalent, see Figure~\ref{F:Cancellation}.

The proof in the general case is similar, hardly more delicate.
\end{proof}

\begin{figure}[htb]
$$\begin{picture}(25,16)(0,2)
\put(0,0){\includegraphics{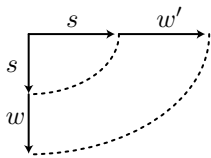}}
\put(-2,12){$\ss$}
\put(-2,5){$\ww$}
\put(6,18){$\ss$}
\put(18,18){$\www$}
\end{picture}$$
\caption{\sf Left-cancellativity: the assumption that $(\ss \ww)\inv (\ss \www)$ reverses to the empty word implies that $\ww\inv \www$ reverses to the empty word, since the first step must consist in deleting~$\ss\inv \ss$.}
\label{F:Cancellation}
\end{figure}

\begin{exam}
\label{X:Main5}
The criterion of Proposition~\ref{P:Cancellation} applies to the monoid~$\MM$ of Example~\ref{X:Preferred}. Indeed, we saw in Example~\ref{X:Main4} that
$$(\tta, \ttb, \ttc, \ttd \mid \tta\ttb = \ttb\ttc = \ttc\tta, \ttb\tta = \ttd\ttb = \tta\ttd, \ttc \tta \tta = \ttd \ttb \ttb)$$
is a complete presentation for~$\MM$. This presentation contains no relation of the form $\ss ... = \ss ...$, so the criterion implies that $\MM$ is left-cancellative. 
Note that Adyan's criterion applies neither to the above presentation, nor to the initially considered presentation of~$\MM$.
\end{exam}

\begin{ques}
How to prove that the monoid~$\MM$ of Examples~\ref{X:Preferred} is (left)-cancellative without using the criterion of Proposition~\ref{P:Cancellation}?
\end{ques}

Combining Proposition~\ref{P:Cancellation} with the completeness criterion of Proposition~\ref{P:Cube} directly leads to the result stated as Theorem~\ref{T:Intro1} in the general introduction. Also, observing that, by definition, a complemented presentation contains no relation $\ss ... = \ss ...$, we deduce

\begin{coro}
Every monoid that admits a complete complemented presentation is left-cancellative. 
\end{coro}

This applies in particular to all Artin--Tits monoids, as first established in~\cite{Gar} in the case of braid monoids and in~\cite{Dlg, BrS} in the general case---and, more generally, to a number of presentations defining Garside monoids (see Section~\ref{S:Garside} below).

\subsection{Existence of least common multiples}

The next application involves least common multiples. We recall that, if $\MM$ is a monoid, and $\xx, \yy$ are elements of~$\MM$, we say that $\yy$ is a \emph{right-multiple} of~$\xx$, or, equivalently, that $\xx$ is a \emph{left-divisor} of~$\yy$, if $\yy = \xx \yy'$ holds for some~$\yy'$. As was pointed out above, subword reversing, when it terminates, produces common right-multiples: assuming that $\MM$ is~$\Mon\SSS\RR$ and that $\ww, \www$ are two words in the alphabet~$\SSS$, reversing~$\ww\inv \www$ leads to a word of the form~$\vvv \vv\inv$ and, in that case, Lemma~\ref{L:Equiv} says that the words~$\ww \vvv$ and $\www \vv$ are $\RR$-equivalent, \ie, they represent a common right-multiple of the elements of~$\MM$ represented by~$\ww$ and~$\www$. 

In the same context, we say that an element~$\zz$ of the monoid~$\MM$ is a \emph{least common right-multiple}, or \emph{right-lcm}, of~$\xx$ and~$\yy$ if $\zz$ is a right-multiple of~$\xx$ and~$\yy$, and every common right-multiple of~$\xx$ and~$\yy$ is a right-multiple of~$\zz$. The following notion has become standard.

\begin{defi} [\bf local lcm]
A monoid~$\MM$ is said to \emph{admit local right-lcm's} if any two elements of~$\MM$ that admit a common right-multiple admit a right-lcm.
\end{defi}

In general, it is uneasy to establish that two elements in a presented monoid possibly admit an lcm, and, therefore, to possibly recognize those monoids that admit local lcm's. This becomes easy whenever a complete presentation is known.

\begin{prop} \cite{Dgp}
\label{P:Lcm}
Assume that $(\SSS, \RR)$ is a complete semigroup presentation. Then a sufficient condition for the monoid~$\Mon\SSS\RR$ to admit local right-lcm's is that $(\SSS, \RR)$ is complemented (in the sense of Definition~\ref{D:Complemented}).
\end{prop}

\begin{proof}[Proof (sketch)]
For~$\ww$ in~$\SSS^*$, let $\clp\ww$ denote the element of the monoid $\Mon\SSS\RR$ represented by~$\ww$. Let $\ww, \www$ belong to~$\SSS^*$ and assume that $\clp\ww$ and~$\clp{\www}$ admit a common right-multiple~$\zz$. This means that there exist words~$\vv, \vvv$ in~$\SSS^*$ such that $\ww \vvv$ and $\www \vv$ both represent~$\zz$, hence are $\RR$-equivalent. As the presentation is complete, $(\ww \vvv)\inv (\www \vv)$ reverses to the empty word. A standard argument shows that the reversing diagram can be split into four subdiagrams as shown in Figure~\ref{F:Lcm}. We deduce that $\ww \cp \www$ and $\www \cp \ww$ exist and that $\zz$ is a right-multiple of the element~$\clp{\ww (\ww \cp \www)}$. By construction, the latter element only depends on~$\clp\ww$ and~$\clp{\www}$. So, every common right-multiple of~$\clp\ww$ and~$\clp{\www}$ is a right-multiple of~$\clp{\ww (\ww \cp \www)}$, which is therefore a right-lcm of~$\clp\ww$ and~$\clp{\www}$.
\end{proof}

\begin{figure}[htb]
$$\begin{picture}(29,19)
\put(0,0){\includegraphics{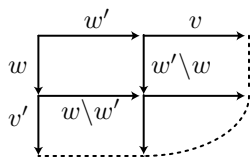}}
\put(-3,12){$\ww$}
\put(-3,5){$\vvv$}
\put(7,18){$\www$}
\put(21,18){$\vv$}
\put(16,12){$\www \cp \ww$}
\put(4,6){$\ww \cp \www$}
\end{picture}$$
\caption{\sf Least common right-multiple: every common right-multiple of the element represented by the words~$\ww$ and~$\www$ is a right-multiple of the element represented by~$\ww(\ww \cp \www)$ and $\www(\www \cp \ww)$.}
\label{F:Lcm}
\end{figure}

Combining Proposition~\ref{P:Lcm} with the completeness criterion of Proposition~\ref{P:Cube} gives now the result stated as Theorem~\ref{T:Intro2} in the general introduction.

\begin{exam}
\label{X:FlagBraid}
The completeness assumption is crucial in Proposition~\ref{P:Lcm}, as shows the example of
\begin{equation}
\label{E:FlagBraid}
\MM = \Mon{\tta,\ttb,\ttc}{\tta\ttb\tta = \ttb^2, \tta\ttc\tta = \ttc\ttb, \ttb\ttc\tta = \ttc^2}.
\end{equation}
Indeed it is easy to see that the relations of~\eqref{E:FlagBraid} provide a presentation of Artin's group braid group~$B_4$---see Example~\ref{X:Braid} below---in terms of the non-standard generators $\tta = \sig1$, $\ttb = \sig2 \sig1$, $\ttc = \sig3 \sig2 \sig1$. The above presentation is complemented: for each pair of generators, there exists exactly one relation in~\eqref{E:FlagBraid} that provides a common right-multiple, and one might think that  right lcm's exist in~$\MM$. However this is \emph{not} the case. Indeed, in~$\MM$, we have $\tta \cdot \ttb\tta = \ttb \cdot \ttb$, but also $\tta \cdot \ttc^2 \ttb = \ttb \cdot \ttc \tta \ttc$, as shows the derivation
$$\tta \ttc^2 \ttb = \tta \ttc \tta \ttc \tta = \ttc \ttb \ttc \tta = \ttc^3 = \ttb \ttc \tta \ttc.$$
Now, $\ttc \tta \ttc$ cannot be a right-multiple of~$\ttb$ in~$\MM$, since no relation of~\eqref{E:FlagBraid} applies to the word~$\ttc \tta \ttc$ and, therefore, the latter cannot be equivalent to a word beginning with the letter~$\ttb$. So, in~$\MM$, the elements~$\tta$ and~$\ttb$ admit no right-lcm. 

The reason for the inapplicability of Proposition~\ref{P:Lcm} is that the presentation~\eqref{E:FlagBraid} is not complete. Indeed, the cube condition fails for the triple $(\tta, \ttb, \ttc)$: the word~$\ttA \ttc \ttC \ttb$ reverses to $\ttc \tta \ttc \tta \ttC \ttA \ttC$, and $(\tta \cdot \ttc \tta \ttc \tta)\inv (\ttb \cdot \ttc \tta \ttc)$, \ie, $\ttA \ttC \ttA \ttC \ttA \ttb \ttc \tta \ttc$, reverses to~$\tta \ttA$, whereas the cube condition would require that it reverses to the empty word. 

Getting a complete presentation for~$\MM$ seems uneasy: after adding the missing relation $\tta \ttc \tta \ttc \tta = \ttb \ttc \tta \ttc$ that fixes the previous obstruction, new obstructions appear. For instance, the cube condition fails for the triple $(\ttb, \ttb, \tta)$, leading to relations of increasing length. It turns out that the elements involved in the new relations need not divide the element~$\ttc^4$, which represent the braid~$\Delta_4^2$ and might be expected to play the role of a fixed point here. 
\end{exam}

\begin{ques}
Does the monoid of~\eqref{E:FlagBraid} admit left-cancellation? Does it embed in its enveloping group?
\end{ques}

Apart from subword reversing, which remains useless as long as no complete presentation has been identified, no general method seems to be eligible here, and even the above natural questions seem to be open. (However we conjecture that $\MM$ is isomorphic to a submonoid of Artin's braid group~$B_4$ via the mapping $\tta \mapsto \sig1$, $\ttb \mapsto \sig2 \sig 1$, and $\ttc \mapsto \sig3 \sig2 \sig1$.)

\subsection{Word problems}
\label{S:WordProblem}

The next application of subword reversing involves word problems. As explained in Remark~\ref{R:WordProblem}, the completeness of~$(\SSS, \RR)$ need not automatically provide a solution for the word problem of the presented monoid~$\Mon\SSS\RR$ or (even less) of the presented group~$\Gr\SSS\RR$, because reversing may never terminate, \ie, there may exist infinite reversing sequences never reaching any terminal word of the form~$\vvv \vv\inv$ with $\vv, \vvv$ in~$\SSS^*$.

\begin{exam}
\label{X:NonTerminate}
Consider the Baumslag--Solitar presentation $(\tta, \ttb \mid \tta^2 \ttb = \ttb \tta)$. Then we find
$$\ttB \tta \ttb \rev \tta \ttB \ttA \ttb \rev \tta \ttB \tta \ttb \ttA,$$
and it is clear that reversing will never terminate since, starting with a signed word~$\sw$, we arrived in two steps at a word that properly includes~$\sw$. Thus, there exists an infinite reversing sequence from~$\sw$ that contains longer and longer words.

Similarly, the type~$\widetilde A_2$ Artin--Tits presentation $$(\tta, \ttb, \ttc \mid \tta\ttb\tta = \ttb\tta\ttb, \ttb\ttc\ttb =\nobreak \ttc\ttb\ttc, \tta\ttc\tta = \ttc\tta\ttc)$$ gives the reversing sequence
$$\ttB \tta\ttc \rev \tta\ttb\ttA\ttB \ttc \rev \tta \ttb \ttA \ttc \ttb \ttC \ttB \rev
\tta \ttb \ttc \tta \ttC \ttA \ttb \ttC \ttB,$$
\ie, we go in three steps from a signed word~$\sw$ to a word that admits~$\sw\inv$ as a proper subword, whence again an infinite reversing sequence.
\end{exam}

We are thus lead to looking for conditions guaranteeing the termination of reversing. First, we have the following characterization, whose proof is similar to that of Proposition~\ref{P:Lcm}.

\begin{lemm}
\label{L:Common}
If $(\SSS, \RR)$ is a complete semigroup presentation, then, for all~$\ww, \www$ in~$\SSS^*$, at least one reversing sequence starting from~$\ww\inv \www$ ends with a word of the form~$\vvv \vv\inv$ with $\vv, \vvv$ in~$\SSS^*$ if and only if the elements of~$\Mon\SSS\RR$ represented by~$\ww$ and~$\www$ admit a common right-multiple.
\end{lemm}

However, Lemma~\ref{L:Common} leads to no practical criterion when we wish to use reversing to establish properties of a still unknown monoid.

Counter-examples to termination like those of Example~\ref{X:NonTerminate}
can occur only when at least one relation involves a word of length~$3$ or more. Indeed, otherwise, the length of the signed words appearing in a reversing sequence does not increase, and termination in guaranteed. Similar results hold for relations involving words of arbitrary length whenever there exists a set of words~$\SSSh$ that includes the original alphabet~$\SSS$ and is closed under reversing in the sense of Proposition~\ref{P:CubeBis} (or its extension to a non-complemented context). Indeed, in this case, the hypothesis means that, in terms of the alphabet~$\SSSh$, every relation involves words of length at most two. Various results can be  established along this line (see~\cite{Dgp}), and we just mention a general one.

\begin{prop} \cite{Dgp}
\label{P:Terminate}
Assume that $(\SSS, \RR)$ is a complete semigroup presentation and there exists a subset~$\SSSh$ of~$\SSS^*$ that includes~$\SSS$ and satisfies the conditions
\begin{gather}
\label{E:Terminate1}
\forall \uu, \uuu \in \SSSh \ \exists \vv, \vvv \in \SSSh \ 
(\uu\inv \uuu \revR \vvv \vv\inv),\\
\label{E:Terminate2}
\forall \uu, \uuu \in \SSSh \ \forall \vv, \vvv \in \SSS^* \ 
(\uu\inv \uuu \revR \vvv \vv\inv \ \Rightarrow \
\vv, \vvv \in \SSSh).
\end{gather}
Then every $\RR$-reversing sequence leads in finitely many steps to a positive--negative word. If $\SSSh$ is finite, then the word problem of the presented monoid~$\Mon\SSS\RR$ is solvable in exponential time, and in quadratic time if $(\SSS, \RR)$ is complemented.
\end{prop}

\begin{proof}[Proof (sketch)]
As shown in Figure~\ref{F:WordPb}, the general form of (the diagram associated with) a signed word~$\sw$ on the alphabet $\SSS \cup \SSS\inv$ is a staircase whose elementary edges belong to~$\SSS$, hence to~$\SSSh$. Then Condition~\eqref{E:Terminate2} implies that every reversing diagram from~$\sw$ splits into a rectangular grid all of which edges belong to~$\SSSh$, whereas Condition~\eqref{E:Terminate1} guarantees that at least one such grid exists.  If the initial word contains $\pp$~letters of~$\SSS$ and $\qq$~letters of~$\SSS\inv$, the grid contains at most $\pp\qq$ squares. If $\SSSh$ is finite, there exist only a finite number of such squares and, therefore, one complete diagram can be constructed in time bounded by~$O(\pp\qq)$.

In order to solve the word problem of the presented monoid~$\Mon\SSS\RR$, starting with two words~$\uu, \uuu$ in the alphabet~$\SSS$, one has to reverse in all possible ways~$\uu\inv \uuu$ and look whether at least one reversing leads to the empty word. Each reversing requires a quadratic amount of time, but there may be exponentially many diagrams and the resulting time upper bound is exponential.

If the presentation is complemented, then there is only one reversing diagram, and, therefore, the overall procedure requires quadratic time only.
\end{proof}

\begin{figure}[htb]
$$\begin{picture}(62,25)
\put(0.5,0.5){\includegraphics{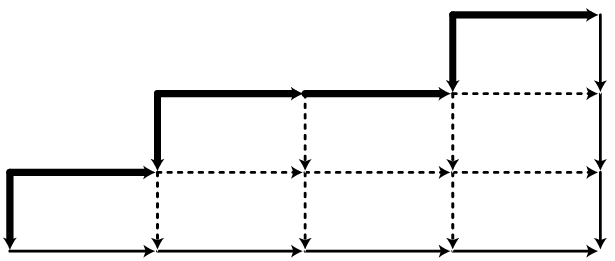}}
\put(30,19.5){$\sw$}
\put(2,4){$\scriptstyle{\in}\SSSh$}
\put(17,4){$\scriptstyle{\in}\SSSh$}
\put(17,12){$\scriptstyle{\in}\SSSh$}
\put(32,4){$\scriptstyle{\in}\SSSh$}
\put(32,12){$\scriptstyle{\in}\SSSh$}
\put(47,4){$\scriptstyle{\in}\SSSh$}
\put(47,12){$\scriptstyle{\in}\SSSh$}
\put(47,20){$\scriptstyle{\in}\SSSh$}
\put(57,4){$\scriptstyle{\in}\SSSh$}
\put(57,12){$\scriptstyle{\in}\SSSh$}
\put(57,20){$\scriptstyle{\in}\SSSh$}
\put(7,2.5){$\scriptstyle{\in}\SSSh$}
\put(7,6.3){$\scriptstyle{\in}\SSSh$}
\put(22,2.5){$\scriptstyle{\in}\SSSh$}
\put(22,6.3){$\scriptstyle{\in}\SSSh$}
\put(22,14.3){$\scriptstyle{\in}\SSSh$}
\put(37,2.5){$\scriptstyle{\in}\SSSh$}
\put(37,6.3){$\scriptstyle{\in}\SSSh$}
\put(37,14.3){$\scriptstyle{\in}\SSSh$}
\put(52,2.5){$\scriptstyle{\in}\SSSh$}
\put(52,6.3){$\scriptstyle{\in}\SSSh$}
\put(52,14.3){$\scriptstyle{\in}\SSSh$}
\put(52,22.3){$\scriptstyle{\in}\SSSh$}
\put(63,13){$\vv$}
\put(30,-2.5){$\vvv$}
\end{picture}$$
\caption{\sf Termination of subword reversing  in Proposition~\ref{P:Terminate}: every edge in the rectangular grid belongs to the subset~$\SSSh$, so the length cannot explode and one reaches a positive--negative word~$\vvv \vv\inv$ in finitely many steps, actually in $O(\pp\qq)$ steps where $\pp$ (\resp $\qq$) is the number of positive (\resp negative) letters in the initial signed word.}
\label{F:Termination}
\end{figure}

Note that, if $(\SSS, \RR)$ is a complemented presentation, then \eqref{E:Terminate1} and~\eqref{E:Terminate2} simply mean that $\SSSh$ is closed under complement in the sense that, for all~$\uu, \uuu$ in~$\SSSh$, the words $\uu \cp \uuu$ and~$\uuu \cp \uu$ exist and belong to~$\SSSh$.

\begin{coro}
\label{C:WordPb}
Under the hypotheses of Proposition~\ref{P:Terminate}, and if, in addition, the monoid $\Mon\SSS\RR$ is right-cancellative, the word problem of the presented  group $\Gr\SSS\RR$ is solvable in exponential time (quadratic time of the presentation is complemented), and the group satisfies a quadratic isoperimetric inequality.
\end{coro}

\begin{proof}
Under the current assumptions, the monoid~$\Mon\SSS\RR$ is cancellative and any two elements admit a common right-multiple. So Ore's conditions are satisfied, and the monoid $\Mon\SSS\RR$ embeds in the group~$\Gr\SSS\RR$, which is a group of fractions~\cite{ClP}. Let $\sw$ be a signed word in the alphabet~$\SSS \cup \SSS\inv$. Then $\sw$ reverses to some positive--negative word $\vvv \vv\inv$, so, using $\eqR$ for the group equivalence, we have
$$\sw \eqR \ew
\quad\Leftrightarrow\quad
\vvv \vv\inv \eqR \ew
\quad\Leftrightarrow\quad
\vv \eqR \vvv
\quad\Leftrightarrow\quad
\vv \eqpR \vvv
\quad\Leftrightarrow\quad
\vv\inv \vvv \revR \ew.$$
This shows that the word problem of~$\Gr\SSS\RR$ can be solved using a double reversing: first reverse~$\sw$ to~$\vvv \vv\inv$, then switch the factors into~$\vv\inv \vvv$ and reverse again; then $\sw$ represents~$1$ in~$\Gr\SSS\RR$ if and only if the second reversing yields the empty word, see Figures~\ref{F:WordPb} and~\ref{F:WordPbBis}.
\end{proof}

\begin{figure}[htb]
$$\begin{picture}(72,28)(0,2)
\put(0.5,0){\includegraphics{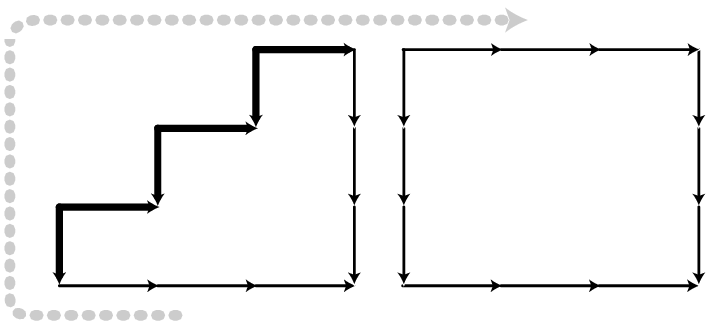}}
\put(18,23){$\sw$}
\put(38,16){$\vv$}
\put(20,1){$\vvv$}
\put(55,31){$\vvv$}
\put(55,1){$\uuu$}
\put(73,16){$\uu$}
\put(24,15){$\revR$}
\put(54,15){$\revR$}
\end{picture}$$
\caption{\sf Solving the word problem of the presented group~$\Gr\SSS\RR$ by a double reversing: starting from~$\sw$, a first reversing leads to~$\vvv \vv\inv$; then copy $\vv\inv$ in front of~$\vvv$ and reverse again: the word~$\sw$ represents~$1$ if and only if the words~$\uu$ and~$\uuu$ are empty.}
\label{F:WordPb}
\end{figure}

\begin{figure}[htb]
$$\begin{picture}(55,27)(0,2)
\put(0.5,0){\includegraphics{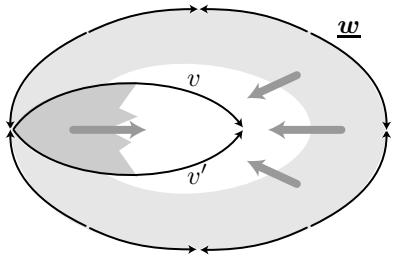}}
\put(45,30){$\sw$}
\put(25,23){$\vv$}
\put(25,10){$\vvv$}
\end{picture}$$
\caption{\sf Corollary~\ref{C:WordPb} viewed as a method for constructing a van Kampen diagram by a double reversing: starting from~$\sw$, the first reversing amounts to filling the space between the outer path~$\sw$ and a positive--negative word~$\vvv \vv\inv$, the second reversing amounts to filling the inner space between~$\vv$ and~$\vvv$.}
\label{F:WordPbBis}
\end{figure}

Restricting to a complemented presentation, and resorting to the counterpart of Proposition~\ref{P:Cancellation} for establishing right-cancellativity, we deduce the result stated as Theorem~\ref{T:Intro3} in the introduction.

\begin{exam}
\label{X:Main6}
Returning once more to the presentation of Example~\ref{X:Preferred}, one easily checks that Proposition~\ref{P:Terminate} applies with $\SSSh = \{\ew, \tta, \ttb, \ttc, \ttd, \tta^2, \tta\ttb, \ttb\tta, \ttb^2\}$. So we deduce that every reversing sequence ends after finitely many steps with a positive--negative word. Thus we obtain a solution for the word problem of the monoid~Ê$\MM$. We saw in Example~\ref{X:Main5} that $\MM$ is left-cancellative. Because of the symmetry of the presentation, $\MM$ is also right-cancellative. So Ore's conditions are satisfied, and we deduce that $\MM$ embeds in a group of fractions, whose word problem can be solved by the double reversing process of Corollary~\ref{C:WordPb}. (It turns out that the involved group of fractions is Artin's $3$-strand braid group~$B_3$.)
\end{exam}

\begin{rema}
\label{R:LD}
If the set~$\SSSh$ involved in Proposition~\ref{P:Terminate} is infinite, which certainly happens if $\SSS$ itself is infinite, then no obvious upper bound exists on the length of the reversing sequences. In~\cite[Chapter~VIII]{Dgd}, there is an example where $\SSS$ is infinite, and the only known upper bound for the length of a reversing sequence starting from~$\ww\inv \www$ where $\ww, \www$ are words of length~$\ell$ (with respect to the alphabet~$\SSS$) is a tower of exponentials whose height is itself exponential in~$\ell$.
\end{rema}

\subsection{Garside structures}
\label{S:Garside}

In the recent years, there has been an increasing interest in a particular class of algebraic structures generically called Garside structures, see for instance~\cite{Bes, Cho, Dfx, Dht, DiM, Kra, McC, Pim}. Several versions exist, but, in this survey, we shall only mention Garside monoids and Garside groups. Our point here is that subword reversing is a useful tool both for recognizing that a presented monoid is a Garside monoid, and for computing in a Garside monoid once one knows it is.

\begin{defi}[\bf Garside] \cite{Dgk}
A monoid~$\MM$ is called \emph{Garside} if it is cancellative, it contains no invertible element except~$1$, any two elements admit a left- and a right-lcm and gcd, and there exists an element~$\Delta$ of~$\MM$ such that the left- and right-divisors of~$\Delta$ coincide, generate~$\MM$, and are finite in number.
\end{defi}

An element~$\Delta$ satisfying the above conditions is called a \emph{Garside element}. By definition, a Garside monoid satisfies Ore's conditions, so it embeds in a group of fractions. A group~$\GG$ is called \emph{Garside} if it is the group of fractions of at least one Garside monoid. The main point about Garside groups and monoids is that the whole structure is fully controlled by the finite lattice consisting of the (left and right) divisors of the Garside element~$\Delta$.

\begin{exam}
\label{X:Braid}
The seminal example of a Garside group is the group~$B_\nn$ of $\nn$-strand braids. For $\nn \ge 2$, the group~$B_\nn$ admits the presentation
\begin{equation}
\label{E:BraidPres}
\bigg\langle \sig1, ..., \sig{n-1} \ \bigg\vert\ 
\begin{matrix}
\sig\ii \sig j = \sig j \sig\ii 
&\text{for} &\vert i-j \vert\ge 2\\
\sig\ii \sig j \sig\ii = \sig j \sig\ii \sig j 
&\text{for} &\vert i-j \vert = 1
\end{matrix}
\ \bigg\rangle.
\end{equation}
A Garside monoid of which $B_\nn$ is a group of fractions is the submonoid~$\BP\nn$ of~$B_\nn$ generated by the elements~$\sig1, ..., \sig\nno$, a Garside element being the so-called fundamental braid $\Delta_\nn$ defined by $\Delta_1 = 1$ and $\Delta_\nn = \Delta_\nno \sig\nno ... \sig2 \sig1$. The lattice of the divisors of~$\Delta_\nn$ in~$\BP\nn$ turns out to be isomorphic to the symmetric group~$\boldsymbol{\mathfrak S}_{\nn}$ equipped with the weak order \cite[Chapter IX]{Eps}---see Figure~\ref{F:Delta}.
\end{exam}

\begin{figure}[htb]
$$\begin{picture}(20,18)
\put(0,0){\includegraphics{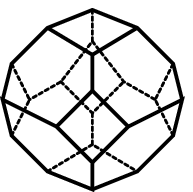}}
\put(9,-3){\small$1$}
\put(8,20){\small$\Delta_4$}
\end{picture}$$
\caption{\sf The $24$-element lattice that controls the Garside structure of the monoid~$\BP4$, topologically a $2$-sphere.}
\label{F:Delta}
\end{figure}

\subsection*{Recognizing Garside structures}

Every Garside monoid (hence every Garside group) admits presentations that are eligible for subword reversing.

\begin{lemm} \cite{Dfx}
\label{L:LcmPresentation}
Every Garside monoid admits a presentation that is complemented and complete with respect to right reversing.
\end{lemm}

\begin{proof}[Proof (sketch)]
Every Garside monoid admits a smallest generating subset, namely the family of its atoms (elements~$\xx$ such that $\xx = \yy \zz$ implies $\yy = 1$ or $\zz = 1$). One obtains a presentation of the expected type by selecting, for each pair of atoms~$(\ss, \sss)$, a relation~$\ss \vvv = \sss \vv$ such that $\ss \vvv$ and $\sss \vv$ represent the right-lcm of~$\ss$ and~$\sss$. (Such a presentation can naturally be called an \emph{lcm-presentation}.)
\end{proof}

It follows from Lemma~\ref{L:LcmPresentation} that, when addressing the question of recognizing Garside monoids from a presentation, it is natural to concentrate on complemented presentations. As for left-cancellation and right-lcm's, the criteria of Propositions~\ref{P:Cancellation} and~\ref{P:Lcm} are relevant, and so are their symmetric counterparts involving the left-reversing procedure of Section~\ref{S:LeftReversing} for right-cancellation and left-lcm's. As for identifying Garside elements, subword reversing still turns out to be suitable. Indeed, when it exists, the least Garside element that is a multiple of the considered generators~$\SSS$ is represented by the longest element in the smallest set of words that includes~$\SSS$ and is closed under the complement and right-lcm operations (Definition~\ref{D:Complement}). We refer to~\cite{Dgk} for details.

\begin{rema}
\label{R:Various}
Recognizing a Garside structure, as well as applying Proposition~\ref{P:Terminate}, or its application Proposition~\ref{T:Intro3}, requires checking a number of conditions. In practice, it may be convenient to work with several presentations of the considered monoid simultaneously, and to appeal to the most convenient one for checking each condition. For instance, in the case of a Garside monoid, a presentation in terms of the atoms is likely to be homogeneous, whereas a presentation in terms of the divisors of a Garside element (``simple elements'') may be more suitable for proving the existence of common multiples. 
\end{rema}

\subsection*{Working in a Garside structure}

The second family of problems consists in investigating a monoid~$\Mon\SSS\RR$ once one knows that this monoid is Garside and that $(\SSS, \RR)$ is a complete complemented presentation. A typical question is to solve the word problem (for the monoid or for the group): here Proposition~\ref{P:Terminate} and Corollary~\ref{C:WordPb} are relevant, since the required hypotheses are necessarily satisfied. Another question is to practically compute the lattice operations associated with the Garside structure, namely the (right)-lcm and the (left)-gcd. As for the right-lcm, Proposition~\ref{P:Lcm} provides a solution, by means of one reversing.  As for the left-gcd, it is easy to show that it can be similarly computed by means of a triple reversing, as will be explained in Corollary~\ref{C:Gcd} below.

Another application is the possibility of using subword reversing to compute the greedy normal form. The latter is a distinguished decomposition of every element into a product of divisors of the considered Garside element. 

\begin{defi}[\bf normal sequence] \cite{ElM, AdjNF, Eps}
Assume that $\MM$ is a Garside monoid with specified Garside element~$\Delta$. A sequence $(\xx_1, ..., \xx_\pp)$ of divisors of~$\Delta$ is called \emph{(right)-normal} if $\xx_1$ is not~$1$ and, for each~$\ii$, the element~$\xx_\ii$ is the maximal divisor of~$\Delta$ that right-divides~$\xx_1 ... \xx_\ii$.\footnote{one says that $\xx$ \emph{right-divides}~$\yy$ if $\yy = \yy' \xx$ holds for some~$\yy'$}
\end{defi}

Every nontrivial element in a Garside monoid admits a unique normal decomposition, a significant result that largely explains the interest in Garside monoids as it entails nice geometric properties for the monoid and the associated group of fractions (automaticity, isoperimetric inequality, ...). Our point here is that subword reversing is closely connected with the computation of the normal form.

\begin{prop} \cite{Dgl}
\label{P:Normal}
Assume that $\Mon\SSS\RR$ is a Garside monoid, that $(\uu_1, ..., \uu_\pp)$ is a sequence of words in~$\SSS^*$ that represents the normal decomposition of some element~$\xx$, and that $\vv$ is a word of~$\SSS^*$ that represents a simple element~$\yy$  left-dividing~$\xx$. Then the normal form for~$\yy\inv \xx$ is represented by the sequence $(\uuu_1, ..., \uuu_\pp)$ inductively determined by $\vv_0 = \vv$ and $\vv_{\ii-1}\inv \uu_\ii \revR \uuu_\ii \vv_\ii\inv$ (see Figure~\ref{F:NormalForm}).
\end{prop}

\begin{figure}[htb]
\begin{picture}(63,10)(0,0)
\put(0,0){\includegraphics{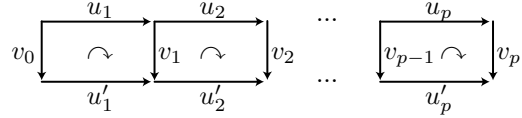}}
\put(7,10){$\uu_1$}
\put(22,10){$\uu_2$}
\put(52,10){$\uu_\pp$}
\put(7,-2){$\uuu_1$}
\put(22,-2){$\uuu_2$}
\put(52,-2){$\uuu_\pp$}
\put(-3,4){$\vv_0$}
\put(16.5,4){$\vv_1$}
\put(31.5,4){$\vv_2$}
\put(46.5,4){$\vv_{\pp-1}$}
\put(61.5,4){$\vv_\pp$}
\put(37.5,1){...}
\put(37.5,9){...}
\put(7,4){$\rev$}
\put(22,4){$\rev$}
\put(54,4){$\rev$}
\end{picture}
\caption{\sf Computation of the normal form by a sequence of reversings.}
\label{F:NormalForm}
\end{figure}

So computing the normal form of~$\yy\inv \xx$ from that of~$\xx$ reduces to a sequence of reversings. Computing the normal form of a product easily follows, as multiplying by~$\yy$ amounts to multiplying by~$\Delta$ and dividing by~$\yy\inv \Delta$, where $\Delta$ is a fixed Garside element. As multiplying by~$\Delta$ is easy, Proposition~\ref{P:Normal} is really the point for computing a normal form. 

A further application of Proposition~\ref{P:Normal} is the computation of  the homology of a Garside monoid~$\MM$ (and of its group of fractions) by using reversing to construct an effective resolution of~$\ZZZZ$ by free $\ZZZZ\MM$-modules---see~\cite{Dgl}, as well as the related papers~\cite{Squ} and~\cite{CMW}.

\subsection{Fractionary decompositions}
\label{S:DoubleReversing}

The next result combines right- and left-reversings to obtain short fractionary decompositions in a group of fractions. 

If $(\SSS, \RR)$ is a presentation which is complemented and such that every reversing sequence is finite, then, for each word~$\sw$ in the alphabet~$\SSS \cup \SSS\inv$, there exist two unique positive words~$\NR(\sw)$ and~$\DR(\sw)$---like ``right-numerator'' and ``right-denominator''---such that $\sw$ reverses to~$\NR(\sw) \DR(\sw)\inv$. These functions make sense at the level of words, but, even when the completeness condition is satisfied, they induce no well defined functions at the level of the presented group~$\Gr\SSS\RR$: if $\sw, \sww$ represent the same element of~$\Gr\SSS\RR$, it need not be the case that $\NR(\sw)$ and~$\NR(\sww)$ are equal, or even equivalent. For instance, if $\ss$ is a letter of~$\SSS$, the words~$\ss \ss\inv$ and~$\ew$ both represent~$1$, but we have $\NR(\ss \ss\inv) = \ss$ and $\NR(\ew) = \ew$, and, in general, $\ss$ does not represent~$1$. This unpleasant phenomenon disappears when both left- and right-reversings are combined.

\begin{prop} [\cite{Dff} in the particular case of braids]
\label{P:Double}
Assume that $(\SSS, \RR)$ is a presentation that is left- and right-complemented, left- and right-complete, and, moreover, such that every left- or right-reversing sequence is finite. Let $\MM = \Mon\SSS\RR$ and $\GG = \Gr\SSS\RR$.

$(i)$ For~$\sw$ a signed word in the alphabet~$\SSS \cup \SSS\inv$, let $\NRL(\sw)$ and~$\DRL(\sw)$ denote the positive words such that $\NR(\sw) \DR(\sw)\inv$ left-reverses to~$\DRL(\sw)\inv \NRL(\sw)$. Then $\NRL$ and $\DRL$ induce well defined mappings of~$\GG$ to~$\MM$. 

$(ii)$ For each signed word~$\sw$ representing in~$\GG$ a fraction~$\xx\inv \xxx$ with $\xx, \xxx$ in~$\MM$, the class of~$\DRL(\sw)$ in~$\MM$ left-divides~$\xx$ and the class of~$\NRL(\sw)$ in~$\MM$ left-divides~$\xxx$.
\end{prop}

Under the above hypotheses, every element of the group~$\GG$ is a fraction~$\xx\inv \xxx$ with $\xx, \xxx$ in~$\MM$. As $\DRL(\sw)\inv \NRL(\sw)$ is the result of reversing~$\sw$ to the right, and then the result to the left, what Proposition~\ref{P:Double} says is that a double reversing process leads to a fractionary decomposition which is minimal among all possible fractionary decompositions in~$\GG$.

\begin{proof}[Proof (sketch)]
If two signed words $\sw, \sww$ in the alphabet~$\SSS \cup \SSS\inv$ satisfy
\begin{equation}
\label{E:Fraction}
\exists \vv, \vvv \in \SSS^*\  \big(\ \NR(\sw) \vv \eqpR \NR(\sww) \vvv
\mbox{\quad and \quad}
\DR(\sw) \vv \eqpR \DR(\sww) \vvv\  \big),
\end{equation}
then, clearly, $\sw$ and $\sww$ represent the same element of the group~$\GG$. Conversely, the relation defined by~\eqref{E:Fraction} can be proved to be an equivalence relation on~$(\SSS \cup \SSS\inv)^*$ that is compatible with left- and right-multiplication. As it includes~$\RR$, it must include the congruence generated by~$\RR$ and, therefore, any two signed words~$\sw, \sww$ representing the same element of~$\GG$ satisfy~\eqref{E:Fraction} (See \cite[Proposition~7.3]{Dgp} for a general form of this result.)

Assume that $\sw, \sww$ represent the same element of~$\GG$. Then \eqref{E:Fraction} holds. Now, if $\NR(\sw) \vv \eqpR \NR(\sww) \vvv$ and $\DR(\sw) \vv \eqpR \DR(\sww) \vvv$ hold, then, by definition of left-reversing, and using $\NL$ for the left-numerator, the counterpart of the right-numerator involving left-reversing, we have the positive word equivalences
\begin{multline*}
\NRL(\sw) 
= \NL(\NR(\sw) \DR(\sw)\inv)
= \NL(\NR(\sw) \vv \vv\inv \DR(\sw)\inv)\\
\eqpR \NL(\NR(\sww) \vvv {\vvv}\inv \DR(\sww)\inv)
= \NL(\NR(\sww) \DR(\sww)\inv)
=\NRL(\sww),
\end{multline*}
and a similar relation for denominators. This proves~$(i)$.

For~$(ii)$, assume that $\xx, \xxx$ lie in~$\MM$ and the signed word~$\sw$ represents~$\xx\inv \xxx$. Let $\uu, \uuu$ be words in the alphabet~$\SSS$ that represent~$\xx$ and~$\xxx$, respectively. Then $\uu\inv \uuu$ and $\NR(\sw) \DR(\sw)\inv$ both represent~$\xx\inv \xxx$, hence $\uu \NR(\sw)$ and $\uuu \DR(\sw)$ represent the same element of~$\GG$. The hypotheses imply that $\MM$ embeds in~$\GG$, so that $\uu \NR(\sw)$ and $\uuu \DR(\sw)$ also represent the same element of~$\MM$. In other words, we have $\uu \NR(\sw) \eqpR \uuu \DR(\sw)$. By definition of left-completeness, this implies that $\uu$ is a left-multiple of~$\DL(\NR(\sw) \DR(\sw)\inv)$, \ie, of~$\DRL(\sw)$, and that $\uuu$ is a left-multiple (with the same quotient) of~$\NL(\NR(\sw) \DR(\sw)\inv)$, \ie, of~$\NRL(\sw)$.
\end{proof}

\begin{coro}
\label{C:Gcd}
Under the assumptions of Proposition~\ref{P:Double}, the left-gcd of two elements~$\xx, \xxx$ of~$\MM$ represented by two words~$\uu, \uu'$ is determined by the following algorithm: right-reverse $\uu\inv \uuu$ to~$\vvv \vv\inv$, then left-reverse~$\vvv \vv\inv$ to~$\ww\inv \www$; finally, left-reverse $\uu \ww\inv$ to~$\ww_*\inv \uu_*$. Then $\ww_*$ must be empty, and $\uu_*$ represents the left-gcd of~$\xx$ and~$\xxx$.
\end{coro}

\begin{proof}
With the notation of Proposition~\ref{P:Double}, we have $\ww = \DRL(\uu\inv \uuu)$ and $\www = \NRL(\uu\inv \uuu)$. Proposition~\ref{P:Double} says that, in the monoid~$\MM$, the class of~$\uu$ is a left-multiple of the class of~$\ww$, the class of~$\uuu$ is a left-multiple of the class of~$\www$, and that the classes of~$\ww$ and~$\www$ admit no nontrivial common left-divisor. It follows that the left-gcd of the classes of~$\uu$ and~$\uuu$ is the class of~$\uu \ww\inv$ (and of $\uuu \www{}\inv$). By construction, this is the class of~$\uu_*$.
\end{proof}

Always in the context of Proposition~\ref{P:Double}, we have \emph{two} different ways of solving the word problem of the presented monoid~Ê$\Mon\SSS\RR$, one using right-reversing, and one using left-reversing. Indeed, for all words~$\ww, \www$ of~$\SSS^*$, we have 
\begin{equation}
\label{E:WordProblemDouble}
\ww \eqpR \www 
\quad\Leftrightarrow\quad 
\ww\inv \www \revR \ew 
\quad\Leftrightarrow\quad 
\www \ww\inv \revlR \ew.
\end{equation}
Associated with these two options are two\footnote{actually four as we may also begin with left-reversing} methods for solving the word problem of the presented group~$\Gr\SSS\RR$: having to consider a signed word~$\sw$, we first right-reverse it to~$\NR(\sw) \DR(\sw)\inv$ but, then, in Proposition~\ref{P:Terminate}, we switch the factors and right-reverse $\DR(\sw)\inv \NR(\sw)$, whereas, in Proposition~\ref{P:Double}, we keep the word and left-reverse it. In the context of Figure~\ref{F:WordPbBis}, the two methods correspond to filling the small inner domain from left to right, or from right to left. In both cases, the criterion is that $\sw$ represents~$1$ if and only if the final word is empty. But the involved final words need not be the same, although they always represent conjugate elements. This leads to the following

\begin{ques}
Assume that $(\SSS, \RR)$ is a complete complemented semigroup presentation. Define $\Phi : \SSS^* \times \SSS^* \to \SSS^* \times \SSS^*$ by $\Phi(\uu, \vv) = (\DR(\uu\inv \vv), \NR(\uu\inv \vv))$. Is every $\Phi$-orbit necessarily finite?
\end{ques}

Experiments suggest a positive answer, at least in the case of braids. If true, this property might be connected with the specific properties of braid conjugacy~\cite{FrG, Geb, GeG}.

\begin{rema}
We already insisted that the left-reversing relation~$\revl$ is not the inverse of the right-reversing relation~$\rev$. However, it may happen that, starting from a negative--positive word~$\uu\inv \uuu$, right-reversing leads to a positive--negative word~$\vvv \vv\inv$, from which left-reversing leads back to the initial word~$\uu\inv \uuu$. But, even in the case of such reversible reversings, it need not be true that every word that can be reached from~$\uu\inv \uuu$ by right-reversing can be reached from~$\vvv \vv\inv$ by left-reversing, and {\it vice versa}. Here is a counter-example in the braid group~$B_4$. We have
$$\siginv3 \siginv1 \sig2 \sig3 \rev \sig2 \sig3 \sig1 \sig2 \siginv1 \siginv3 \siginv2 \siginv1 \revl \siginv3 \siginv1 \sig2 \sig3.$$
Now, $\siginv3 \sig2 \sig3 \sig1 \sig2 \siginv3 \siginv2 \siginv1$ can be reached from $\sig2 \sig3 \sig1 \sig2 \siginv1 \siginv3 \siginv2 \siginv1$ using~$\revl$, but it cannot be reached from $\siginv3 \siginv1 \sig2 \sig3$ using~$\rev$.
\end{rema}

\subsection{An embeddability criterion}
\label{S:Mixed}

We conclude the section with a (partial) criterion guaranteeing that a monoid~$\Mon\SS\RR$ possibly embeds in its enveloping group~$\Gr\SS\RR$. Again we resort to a combination of right- and left-reversings.

If a semigroup presentation~$(\SSS, \RR)$ satisfies the assumptions of Proposition~\ref{P:Double}, a signed word~$\sw$ represents~$1$ in the associated group if and only if a double reversing from~$\sw$, namely a right-reversing followed by a left-reversing, leads to the empty word. A fortiori, if $\revdR$ denotes the transitive closure of the union of the relations~$\revR$ and~$\revlR$, \begin{equation}
\label{E:Reducing}
\mbox{A signed word~$\sw$ represents~$1$ in~$\Gr\SSS\RR$ if and only if $\sw \revdR \ew$ holds.}
\end{equation}
On the other hand, for the trivial presentation of a free group, \eqref{E:Reducing} holds as well, since, in this case, the relation~$\revd$ is just the standard free group reduction. We may thus wonder whether \eqref{E:Reducing} holds for more general families of presentations, typically for all Artin--Tits presentations. It is easy to see that this is not the case.

\begin{exam}
\label{X:CounterArtin}
Let $\GG$ be the right-angled Artin--Tits group defined by
\begin{equation}
\label{E:CounterArtin}
\GG = \Gr{\tta, \ttb, \ttc, \ttd}{\tta \ttc = \ttc \tta, \ttb \ttc = \ttc \ttb, \tta \ttd = \ttd \tta, \ttb \ttd = \ttd \ttb}.
\end{equation}
So $\GG$ is a direct product of two free groups of rank~$2$. As mentioned in Example~Ê\ref{X:ArtinCube}, the presentation of~\eqref{E:CounterArtin} is right- and left-complete, the associated monoid is cancellative, it embeds in the group~$\GG$ \cite{Par}, and satisfies lots of regularity properties. Let $\sw = \ttA \ttC \ttd \tta \ttB \ttD \ttc \ttb$, \ie, $\tta\inv \ttc\inv \ttd \tta \ttb\inv \ttd\inv \ttc \ttb$. As shown in Figure~\ref{F:Counter} (left), $\sw$ represents~$1$ in~$\GG$, but $\sw \revd \ew$ is false, as $\sw$ is eligible for no right- or left-reversing associated with the relations of~\eqref{E:CounterArtin}. 
\end{exam}

A more interesting relation appears when, in addition to right- and left-reversing, we also allow applying the semigroup relations and their inverses.

\begin{defi}[\bf mixed reversing]
For $(\SSS, \RR)$ a semigroup presentation, the \emph{mixed reversing} relation~$\revmR$ is  the transitive closure of the union of~$\revR$, $\revlR$, $\RR$, and~$\RR\inv$.
\end{defi}

Thus, a signed word~$\sww$ is obtained from another signed word~$\sw$ by one step of mixed reversing if we have $\sw = \su \, \sv \, \suu$ and $\sww = \su \, \svv \, \suu$ with

- either $\sv = \ss\inv \sss$ and $\svv = \vvv \vv\inv$ for some relation $\ss \vvv = \sss \vv$ of~$\RR$,

- or $\sv = \sss \ss\inv$ and $\svv = \vv\inv \vvv$ for some relation $\vvv \ss = \vv \sss$ of~$\RR$,

- or $\sv = \svv$ is a relation of~$\RR$, 

- or we have $\sv = \vv\inv$ and $\svv = \vvv{}\inv$ for some relation $\vv = \vvv$ of~$\RR$. 

\noindent By construction, the relation~$\revmR$ is included in the congruence generated by~$\RR$: if $\sw \revmR \sww$ holds, then $\sw$ and~$\sww$ represent the same element of the group~$\Gr\SSS\RR$. 

The mixed reversing relation naturally occurs in the case of braids when investigating the handle reduction of~\cite{Dfo}. In general, the relation~$\revmR$ properly includes the double reversing relation~$\revdR$. For instance, with respect to mixed reversing, the word~$\sw$ of Example~\ref{X:CounterArtin} reduces to the empty word:
\begin{multline*}
\ttA \ttC \ttd \tta \ttB \ttD \ttc \ttb
\revmR  \ttC \ttA \ttd \tta \ttB \ttD \ttc \ttb
\revmR  \ttC \ttA \tta \ttd \ttB \ttD \ttc \ttb
\revmR  \ttC \ttd \ttB \ttD \ttc \ttb\\
\revmR  \ttC \ttd \ttD \ttB \ttc \ttb
\revmR  \ttC \ttd \ttD \ttB \ttb \ttc
\revmR  \ttC \ttd \ttD \ttc
\revmR  \ttC \ttc
\revmR \ew.
\end{multline*}

We are thus led to considering the condition
\begin{equation}
\label{E:ReducingBis}
\mbox{A signed word~$\sw$ represents~$1$ in~$\Gr\SSS\RR$ if and only if $\sw \revmR \ew$ holds.}
\end{equation}
Saying that \eqref{E:ReducingBis} is valid may be seen as claiming the existence of a weak form of Dehn's algorithm \cite{LyS} inasmuch as it means that, if a signed word~$\sw$ represents~$1$, then it can be transformed into the empty word without introducing any pair~$\ss\inv \ss$ or~$\ss\inv \ss$. Geometrically, \eqref{E:ReducingBis} means that, if $\sw$ represents~$1$, then there exists a van Kampen diagram whose boundary is labeled~$\sw$ and which contains (at least) one tile containing two adjacent letters of~$\sw$ with opposite signs, or adjacent letters of~$\sw$ forming one half of a relation of~$\RR$ (Figure~\ref{F:Mixed}).

\begin{figure}[htb]
\begin{picture}(110,22)(0,1)
\put(-1,-1){\includegraphics{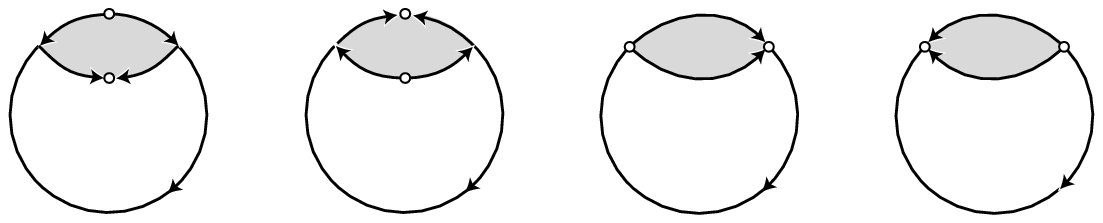}}
\put(19,2){$\sw$}
\put(49,2){$\sw$}
\put(79,2){$\sw$}
\put(109,2){$\sw$}
\put(8,17){$\revR$}
\put(1,18){$\vv$}
\put(17.5,18){$\vvv$}
\put(38,16){$\revlR$}
\put(31,18){$\vv$}
\put(47.5,18){$\vvv$}
\put(68,16){$\eqpR$}
\put(69,21){$\vv$}
\put(69,10.5){$\vvv$}
\put(98,16){$\eqmR$}
\put(99,21){$\vv$}
\put(99,10.5){$\vvv$}
\end{picture}
\caption{\sf Saying that $\sw \revmR \ew$ holds means that there is a van Kampen diagram with boundary~$\sw$ that contains (at least) one tile $\vv = \vvv$ such that $\sw$ contains the middle two letters of~$\vv\inv \vvv$ (\resp the middle two letters of~$\vv \vvv{}\inv$, \resp all of~$\vv$, \resp all of~$\vv\inv$).}
\label{F:Mixed}
\end{figure}

We observed above that the hypothesis that the equivalence~\eqref{E:Reducing} need not be true\footnote{nor does either the naive version, relations of~$\RR$ plus free group reduction: in the free Abelian group generated by~$\tta, \ttb, \ttc$, the word $\tta \ttB \ttc \ttA \ttb \ttC$ represents~$1$ but it is eligible neither for a positive commutation relation, nor for a free group reduction}. It is easy to see that \eqref{E:ReducingBis} may also fail.

\begin{exam}
\label{X:Counter}
Consider
\begin{equation}
\label{E:Counter}
\GG = \Gr{\tta, \ttb, \ttc, \ttd, \tte, \ttf}{\tta \ttc = \ttc \tta \tte, \ttb \ttc = \ttc \ttb \tte, \tta \ttd = \ttd \tta \ttf, \ttb \ttd = \ttd \ttb \ttf}.
\end{equation}
In other words, $\GG$ admits the presentation $\Gr{\tta, \ttb, \ttc, \ttd}{[\tta, \ttc] =[\ttb, \ttc], [\tta, \ttd] = [\ttb, \ttd]}$. As in Example~\ref{X:CounterArtin}, put $\sw = \ttA \ttC \ttd \tta \ttB \ttD \ttc \ttb$. As shown in Figure~\ref{F:Counter} (right), $\sw$ represents~$1$ in~$\GG$, but $\sw \revm \ew$ is false, as $\sw$ is eligible neither for a right- or left-reversing, nor for a positive or negative relation. It can be checked that the presentation~\eqref{E:Counter} is complete. 
\end{exam}

\begin{figure}[htb]
\begin{picture}(90,24)(0,-1)
\put(-1,-1){\includegraphics{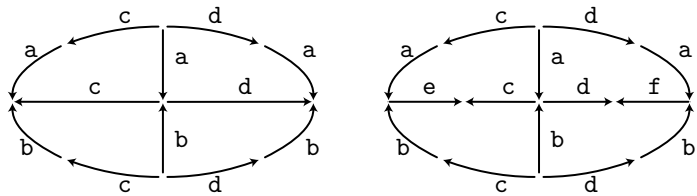}}
\put(21.5,15){$\tta$}
\put(21.5,4){$\ttb$}
\put(10,11){$\ttc$}
\put(1.5,16){$\tta$}
\put(14,-2){$\ttc$}
\put(1,3){$\ttb$}
\put(14,20.5){$\ttc$}
\put(30,11){$\ttd$}
\put(38.5,16){$\tta$}
\put(26,-2){$\ttd$}
\put(39,3){$\ttb$}
\put(26,20.5){$\ttd$}

\put(71.5,15){$\tta$}
\put(71.5,4){$\ttb$}
\put(54.5,11){$\tte$}
\put(65,11){$\ttc$}
\put(51.5,16){$\tta$}
\put(64,-2){$\ttc$}
\put(51,3){$\ttb$}
\put(64,20.5){$\ttc$}
\put(75,11){$\ttd$}
\put(84.5,11){$\ttf$}
\put(88.5,16){$\tta$}
\put(76,-2){$\ttd$}
\put(89,3){$\ttb$}
\put(76,20.5){$\ttd$}
\end{picture}
\caption{\sf Two van Kampen diagrams showing that the word $\ttA \ttC \ttd \tta \ttB \ttD \ttc \ttb$ represents~$1$ in the group of Example~\ref{X:CounterArtin} (left) and in that of Example~\ref{X:Counter} (right): the latter diagram contains no face of the form considered in Figure~\ref{F:Mixed}.}
\label{F:Counter}
\end{figure}

However, it seems difficult to construct examples of the above type with Artin--Tits presentations, because each Artin relation $\ss \sss \ss ... = \sss \ss \sss ...$ is fully determined by any pair of adjacent letters. This makes the 
following conjecture plausible.

\begin{conj}
\label{C:Artin}
Every Artin--Tits presentation~$(\SSS, \RR)$ satisfies~\eqref{E:ReducingBis}.
\end{conj}

As in the proof of Proposition~\ref{P:Double}, the hard part for establishing~\eqref{E:ReducingBis} is to show that the relation $\su\inv \sv \revmR \ew$ is transitive. A natural approach would consist in showing that $\revmR$ is confluent, this meaning that, if we have $\sw \revmR \sw_1$ and $\sw \revmR \sw_1$, then $\sw_1 \revmR \sww$ and $\sw_2 \revR \sww$ hold for some~$\sww$. 

A possible proof of Conjecture~\ref{C:Artin} might entail an extension to arbitrary Artin--Tits groups of the handle reduction algorithm of~\cite{Dfo} that would also solve the word problem. But, in general, a proof of~\eqref{E:ReducingBis} alone need not solve the word problem, since a signed word can lead to infinitely many words under mixed reversing, even in the Garside case (see an example in~\cite{Dhg} involving $4$-strand braids). By contrast, \eqref{E:ReducingBis} is sufficient to solve the embeddability problem, as we shall now explain.

It is well known that left- and right-cancellativity are necessary conditions for a monoid to embed in a group, but that these conditions are not sufficient, as shows the example of
$$\Mon{\tta, \ttb, \ttc, \ttd, \tta', \ttb', \ttc', \ttd'}{\tta \ttc = \tta' \ttc' , \tta \ttd = \tta' \ttd', \ttb \ttc = \ttb' \ttc'},$$
where the relation $\ttb \ttd = \ttb' \ttd'$ fails in the monoid, but holds in every group that satisfies the above relations. As recalled in the proof of Corollary~\ref{C:WordPb}, assuming the existence of common multiples is sufficient to guarantee the embeddability of the monoid in a group of fractions. But, apart from this special case, very few embeddability criteria are known. Here is the point where mixed reversing might prove useful.

\begin{prop}
\label{P:Embeddability}
If $(\SSS, \RR)$ is a complete semigroup presentation that satisfies~\eqref{E:ReducingBis}, then the monoid~$\Mon\SSS\RR$ embeds in the group~$\Gr\SSS\RR$.
\end{prop}

In the result above, completeness refers to right-reversing. Of course, the same conclusion holds under the symmetric hypothesis involving left-reversing. 

Proposition~\ref{P:Embeddability} will follow from controlling particular decompositions for a signed word. 

\begin{defi} [\bf bridge]
(Figure~\ref{F:Bridge}) Let $(\SSS, \RR)$ be a semigroup presentation and $\uu, \vv$ be words in the alphabet~$\SSS$. We say that a signed word~$\sw$ is an \emph{$\RR$-bridge} from~$\uu$ to~$\vv$ if there exists a sequence of positive words $(\uu_1, \vv_1, \ww_1, ..., \uu_\pp, \vv_\pp, \ww_\pp)$ satisfying
\begin{gather}
\label{E:Bridge1}
\sw = \uu_1 \vv_1\inv \uu_2 \vv_2\inv ... \uu_\pp \vv_\pp\inv,\\
\label{E:Bridge2}
\uu \eqpR \uu_1 \ww_1, \quad \vv_1 \ww_1 \eqpR \uu_2 \ww_2, \quad ...\quad ,  \vv_{\pp-1} \ww_{\pp-1} \eqpR \uu_\pp \ww_\pp, \quad
\vv_\pp \ww_\pp \eqpR \vv.
\end{gather}
\end{defi}

\begin{figure}[htb]
$$\begin{picture}(62,41)(0,1)
\put(0,0){\includegraphics{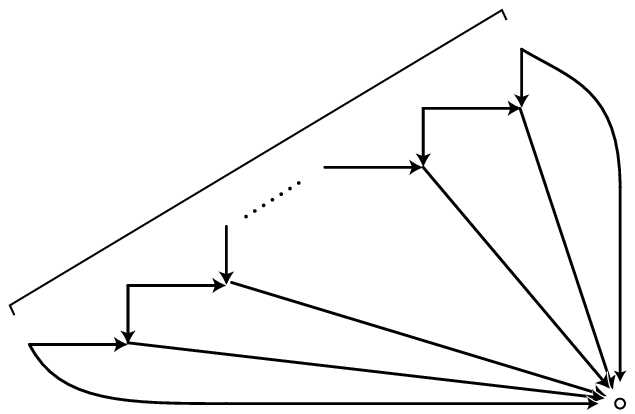}}
\put(6,8){$\uu_1$}
\put(16,14){$\uu_2$}
\put(34,26.5){$\uu_{\pp-1}$}
\put(46,32.5){$\uu_\pp$}
\put(13.5,10){$\vv_1$}
\put(23.5,16){$\vv_2$}
\put(43.5,28){$\vv_{\pp-1}$}
\put(53.5,33.5){$\vv_\pp$}
\put(4,1){$\uu$}
\put(30,6){$\ww_1$}
\put(36,10){$\ww_2$}
\put(49,18){$\ww_{\pp-1}$}
\put(56,22){$\ww_\pp$}
\put(60,34){$\vv$}
\put(24,28){$\sw$}
\end{picture}$$
\caption{\sf An $\RR$-bridge~$\sw$ from~$\uu$ to~$\vv$: a word equivalent to~$\uu \vv\inv$, plus a collection of $\eqpR$-commutative diagrams connecting~$\uu$ to~$\vv$ through~$\sw$.}
\label{F:Bridge}
\end{figure}

If $\sw$ is an $\RR$-bridge from~$\uu$ to~$\vv$, then the relations~\eqref{E:Bridge1} and~\eqref{E:Bridge2} easily imply that $\sw$ and $\uu \vv\inv$ represent the same element in the group~$\Gr\SSS\RR$. For our current purpose, the nice point is that bridges are preserved under mixed reversing.

\begin{lemm}
\label{L:Mixed}
If $(\SSS, \RR)$ is a complete semigroup presentation and $\sw$ is an $\RR$-bridge from~$\uu$ to~$\vv$, then every word~$\sww$ satisfying $\sw \revmR \sww$ is an $\RR$-bridge from~$\uu$ to~$\vv$ as well. 
\end{lemm}

\begin{proof} 
We assume that $(\uu_1, \vv_1, \ww_1, ..., \uu_\pp, \vv_\pp, \ww_\pp)$ is a sequence witnessing that $\sw$ is an $\RR$-bridge from~$\uu$ to~$\vv$, and we shall construct a sequence witnessing that $\sww$ is also an $\RR$-bridge from~$\uu$ to~$\vv$. Without loss of generality we may assume that $\sww$ is obtained by one elementary step of mixed reversing from~$\sw$. 

Case 1: The word~$\sww$ is obtained from~$\sw$ by applying one relation of~$\RR$ (\resp$\RR\inv$). We may assume that none of the intermediate words~$\vv_1, \uu_2, \vv_2, ..., \vv_{\pp-1}, \uu_\pp$ is empty, for, otherwise, we may gather adjacent words~$\uu_\kk$ or~$\vv_\kk$. Then, the intermediate words~$\uu$ and~$\vv$ are nonempty implies that the subword of~$\sw$ involved in the transformation is a subword of some factor~$\vv_\kk$ (\resp $\uu_\kk$). Define $\vv'_\kk$ (\resp $\uu'_\kk$) to be the result of applying the involved relation in~$\vv_\kk$ (\resp $\uu_\kk$), and $\uu'_\ii, \vv'_\ii, \ww'_\ii$ to be equal to $\uu_\ii, \vv_\ii, \ww_\ii$ in all other cases. Then $(\uu'_1, ..., \ww'_\pp)$ is the expected witness. 

Case 2: (Figure~\ref{F:BridgeRev} top) The word~$\sww$ is obtained from~$\sw$ by applying one step of left-reversing. By definition, there exists an index~$\kk$, letters~$\ss, \sss$ in~$\SSS$, and a relation $\vvv \ss =  \vv \sss$ of~$\RR$  such that $\uu_\kk$ ends with~$\ss$, $\vv_\kk$ ends with~$\sss$, and $\sww$ is obtained from~$\sw$ by replacing the corresponding subword~$\sss \ss\inv$ with~$\vv\inv \vvv$. Define $\uuu_\kk, ... \www_{\kk+1}$ by $\uu_\kk = \uuu_\kk \ss$, $\vvv_\kk = \vvv$, $\uuu_{\kk+1} = \vv$, $\vv_\kk = \vvv_{\kk+1} \sss$, $\www_\kk = \ss \ww_\kk$, $\www_{\kk+1} = \sss \ww_\kk$, and complete with $\uuu_\ii = \uu_\ii$ for $\ii < \kk$ and $\uuu_{\ii+1} = \uu_\ii$ for $\ii > \kk$, and similarly for $\vvv_\ii$ and~$\www_\ii$. Then $(\uu_1, \vv_1, \ww_1, ..., \uuu_{\pp+1}, \vvv_{\pp+1}, \www_{\pp+1})$ is the expected witness.

Case 3: (Figure~\ref{F:BridgeRev} bottom) The word~$\sww$ is obtained from~$\sw$ by applying one step of right-reversing. By definition, there exists an index~$\kk$, letters~$\ss, \sss$ in~$\SSS$, and a relation $\ss \vvv = \sss \vv$ of~$\RR$  such that $\vv_{\kk-1}$ begins with~$\sss$, $\uu_\kk$ begins with~$\ss$, and $\sww$ is obtained from~$\sw$ by replacing the corresponding subword~$\ss\inv \sss$ with~$\vvv \vv\inv$. Define $\uuu_{\kk-1}, ... \uuu_{\kk+1}$ by $\vv_{\kk-1} = \ss \vvv_{\kk-1}$, $\uuu_\kk = \vv$, $\vvv_{\kk} = \vvv$, $\uu_\kk = \sss \uuu_{\kk+1}$, and complete with $\uuu_\ii = \uu_\ii$ for $\ii < \kk$ and $\uuu_{\ii+1} = \uu_\ii$ for $\ii > \kk$, and similarly for $\vvv_\ii$ and~$\www_\ii$. Here is the key point. By hypothesis, we have $\ss \vvv_{\kk-1} \www_{\kk-1}  \eqpR \sss \uuu_{\kk+1} \www_{\kk+1}$. As the presentation~$(\SSS, \RR)$ is complete with respect to left-reversing, there must exist a word~$\www_\kk$ satisfying 
$$\uuu_\kk \www_\kk \eqpR \vvv_{\kk-1} \www_{\kk-1} 
\quad\mbox{and}\quad
\vvv_\kk \www_\kk \eqpR \uuu_{\kk+1} \www_{\kk+1}.$$
Then $(\uu_1, \vv_1, \ww_1, ..., \uuu_{\pp+1}, \vvv_{\pp+1}, \www_{\pp+1})$ is  the expected witness.
\end{proof}

\begin{figure}[htb]
\begin{picture}(106,32)(2,0)
\put(0,0){\includegraphics{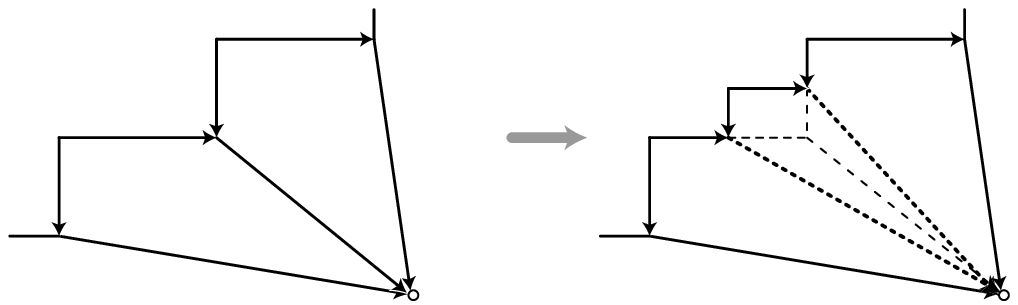}}
\put(20,5){$\ww_{\kk-1}$}
\put(30,11){$\ww_\kk$}
\put(40,16){$\ww_{\kk+1}$}
\put(7,11){$\vv_{\kk-1}$}
\put(12,18){$\uu_{\kk}$}
\put(23,22){$\vv_{\kk}$}
\put(26,29){$\uu_{\kk+1}$}
\put(80,5){$\www_{\kk-1}$}
\put(85,11){$\www_\kk$}
\put(90,14){$\www_{\kk+1}$}
\put(100,16){$\www_{\kk+2}$}
\put(66.5,11){$\vvv_{\kk-1}$}
\put(68,18.5){$\uuu_{\kk}$}
\put(74.5,19){$\vvv_{\kk}$}
\put(74,23.5){$\uuu_{\kk+1}$}
\put(82.5,24){$\vvv_{\kk+1}$}
\put(86,28.5){$\uuu_{\kk+2}$}
\end{picture}
\begin{picture}(92,27)(0,2)
\put(0,0){\includegraphics{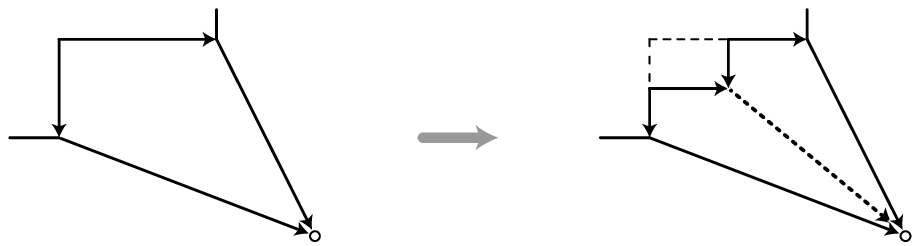}}
\put(13,4){$\ww_{\kk-1}$}
\put(27,12){$\ww_\kk$}
\put(6.5,15){$\vv_{\kk-1}$}
\put(12,22){$\uu_{\kk}$}
\put(73,4){$\www_{\kk-1}$}
\put(87,12){$\www_{\kk+1}$}
\put(82,10){$\www_{\kk}$}
\put(66.5,13){$\vvv_{\kk-1}$}
\put(68,17){$\uuu_{\kk}$}
\put(74.5,18){$\vvv_{\kk}$}
\put(74,22.5){$\uuu_{\kk+1}$}
\end{picture}
\caption{\sf Applying one step of left-reversing (top) or right-reversing (bottom) creates in general a new commutative diagram; the point is that, in the case of right-reversing, completeness guarantees the existence of the factorizing edge~$\www_\kk$.}
\label{F:BridgeRev}
\end{figure}

We can now establish Proposition~\ref{P:Embeddability}.

\begin{proof}[Proof of Proposition~\ref{P:Embeddability}]
The point is to establish that, if two positive words~$\uu, \vv$ represent the same element in the group~$\Gr\SSS\RR$, then they also represent the same element in the monoid~$\Mon\SSS\RR$, \ie, that $\uu \eqpR \vv$ holds. So assume that $\uu, \vv$ represent the same element in the group. Then $\uu \vv\inv$ represents~$1$ in the group. So, by hypothesis, we have $\uu \vv\inv \revmR \ew$. 

Next, we observe that $\uu \vv\inv$ is an $\RR$-bridge from~$\uu$ to~$\vv$, as witnesses the sequence~$(\uu, \vv, \ew)$: indeed, in this case, \eqref{E:Bridge2} reduces to the valid statements $\uu \eqpR \uu \ew$ and $\vv \ew \eqpR \vv$. 

By Lemma~\ref{L:Mixed}, we deduce that the empty word is an $\RR$-bridge from~$\uu$ to~$\vv$. Assume that $(\uu_1, \vv_1, \ww_1, ..., \uu_\pp, \vv_\pp, \ww_\pp)$ is a witness-sequence. Then \eqref{E:Bridge1} implies $\uu_1 = \vv_1 = ... = \uu_\pp = \vv_\pp = \ew$, and, therefore, \eqref{E:Bridge2} reads
$$\uu \eqpR \ww_1 \eqpR \ww_2 \eqpR ... \eqpR \ww_\pp \eqpR \vv,$$
so $\uu \eqpR \vv$ holds, as expected.
\end{proof}

If Conjecture~\ref{C:Artin} is true, applying Proposition~\ref{P:Embeddability} would provide an alternative proof of the embeddability of every Artin--Tits monoid in the associated group, arguably more natural than the beautiful but indirect argument of~\cite{Par} based on the existence of certain linear representations extending the Lawrence--Krammer representation of braids~\cite{Krb}.

\goodbreak

Finally, it is easy to deduce from Lemma~\ref{L:Mixed} one more result involving mixed reversing. 

\begin{prop}
\label{P:MixedBis}
Assume that $(\SSS, \RR)$ is a complete semigroup presentation, and we have $\sw \revmR \sww$ with $\sw = \uu \vv\inv$ and $\uu, \vv \in\SSS^*$. Then we have $\sww \revR\nobreak \uuu \vvv{}\inv$ for some~$\uuu, \vvv$ in~$\SSS^*$ satisfying $\uu \eqpR \uu' \ww$ and $\vv \eqpR \vv' \ww$ for some~$\ww$ in~$\SSS^*$.
\end{prop}

\begin{proof}[Proof (sketch)]
Using the characterization of completeness given in Lemma~\ref{L:Complete}, one easily shows that, if $\sww$ is an $\RR$-bridge from~$\uu$ to~$\vv$, then we have $\sww \revR\nobreak \uuu \vvv{}\inv$ for some words~$\uuu, \vvv$ in~$\SSS^*$ satisfying $\uu \eqpR \uu' \ww$ and $\vv \eqpR \vv' \ww$ for some~$\ww$ in~$\SSS^*$. Now, by Lemma~\ref{L:Mixed}, $\uu \vv\inv \revmR\nobreak \sww$ implies that $\sww$ is an $\RR$-bridge from~$\uu$ to~$\vv$. 
\end{proof}

This result answers a question implicit in Remark~\ref{R:LeftRev}: we observed that $\sw \revlR \sww$ need not imply $\sww \revR \sw$, but Proposition~\ref{P:MixedBis} shows that, if $\sw$ is positive--negative, then $\sw \revlR \sww$ implies $\sww \revR \swww$ for some~$\swww$ that is connected with~$\sw$ simply.

\section{Subword reversing: efficiency}
\label{S:Efficiency}

In Section~\ref{S:Description}, subword reversing was introduced as a strategy for constructing van Kampen diagrams. When this strategy works, \ie, when the considered presentation happens to be complete, it is natural to bring the quality of this strategy into question. We shall see below that subword reversing need not be optimal, but that some explicit bounds exist on the lack of optimality. More generally, we gather here a few results about the algorithmic complexity of subword reversing, both in the general case and in the specific case of Artin's presentation of braid groups, which is a key example.

\subsection{Upper bounds}

For each semigroup presentation, there is a natural notion of distance between equivalent words, and there is a similar notion for each particular strategy constructing derivations between equivalent words.

\begin{defi}[\bf distances]
$(i)$ If $(\SSS, \RR)$ is a semigroup presentation and $\ww, \www$ are $\RR$-equivalent words of~$\SSS^*$, the \emph{combinatorial distance}~$\distR(\ww, \www)$ is the minimal number of relations of~$\RR$ relations needed to transform~$\ww$ into~$\www$.

$(ii)$ If, moreover, $(\SSS, \RR)$ is complete, we define $\distRrev(\ww, \www)$ to be the minimal number of nontrivial\footnote{a reversing step of the form $\ss\inv \ss \rev \ew$ is called \emph{trivial}; all other reversing steps are called \emph{nontrivial}} steps needed to reverse $\ww\inv \www$ into the empty word.
\end{defi}

By definition, we have
\begin{equation}
\label{E:Distances}
\distR(\ww, \www) \le \distRrev(\ww, \www)
\end{equation}
for all pairs of $\RR$-equivalent words~$\ww, \www$, and saying that the reversing stretegy is optimal would mean that \eqref{E:Distances} is an equality. This need not be true in general.

\begin{exam}
\label{X:NotOptimal}
Let us consider Artin's presentation~\eqref{E:BraidPres} of the $4$-strand braid group~$B_4$, and the two words $\sig1 \sig2 \sig1 \sig3 \sig2 \sig1$ and $\sig3 \sig2 \sig3 \sig1 \sig2 \sig3$, which both represent the braid~$\Delta_4$. Then the combinatorial distance turns out to be~$6$, whereas $8$~reversing steps are needed to reverse the quotient into the empty word~\cite{Dhx}. So, in this case, the reversing strategy is not optimal, \ie, it does not provide a shortest derivation between the two words, or, equivalently, a van Kampen diagram with the minimal number of faces.
\end{exam}

However, it turns out that the gap between the combinatorial distance and the reversing distance cannot be arbitrarily large. Indeed, without assuming anything about the termination of reversing, one has the following nontrivial result. Hereafter, we denote by $\vert\ww\vert$ the length (number of letters) of a word~$\ww$.

\begin{prop} \cite{Dgc}
\label{P:DoubleExp}
Assume that $(\SSS, \RR)$ is a finite, complete, complemented presentation, and, moreover, that the relations of $\RR$ preserve the length\footnote{\ie, they are of the form $\vv = \vvv$ with $\vert\vv\vert = \vert\vvv\vert$}. Then there exists a constant~$\CC$ such that, for all $\RR$-equivalent words $\ww, \www$, one has
\begin{equation}
\distR(\ww, \www) \le \distRrev(\ww, \www) \le \distR(\ww, \www) \cdot 2^{2^{\CC\vert \ww \vert}}.
\end{equation}
\end{prop}

The constant~$\CC$ mentioned in Proposition~\ref{P:DoubleExp} can be computed effectively: roughly speaking, it measures the maximal number of relations involved in a cube condition for a triple of letters of~$\SSS$. The reason why a double exponential appears is not yet clear, nor is either the possibility of extending the result to a non-complemented or non-length-preserving context.

Stronger results exist in particular cases. In the context of Proposition~\ref{P:Terminate}, we observed that, in the complemented case and when there exists a finite set of words that is closed under complement, the existence of grids such as the one of Figure~\ref{F:Termination} provides a quadratic upper bound for the number of reversing steps: there exists a constant~$\CC$ such that, for all $\RR$-equivalent words~$\ww, \www$, we have
\begin{equation}
\label{E:Quadratic}
\distRrev(\ww, \www) \le \CC \cdot \vert \ww\vert \cdot \vert \www\vert,
\end{equation}
where $\CC$ is a constant that can be computed explicitely from the presentation. 

Actually, a stronger result holds, as there is no need that the words~$\ww, \ww'$ be $\RR$-equivalent. 

\begin{defi}[\bf complexity]
If $(\SSS, \RR)$ is a complemented presentation, and~$\ww, \www$ are words of~$\SSS^*$, the \emph{reversing complexity}~$\compl(\ww, \www)$ of~$(\ww, \www)$ is the number of nontrivial steps needed to reverse~$\ww\inv \www$ to a positive--negative word, if it exists.
\end{defi}

If $\ww$ and~$\www$ are $\RR$-equivalent, $\compl(\ww, \www)$ coincides with~$\distRrev(\ww, \www)$ but, in general, $\compl(\ww, \www)$ is $\distRrev(\ww\vvv, \www\vv)$, where $\vv, \vvv$ are the positive words such that $\ww\inv \www$ reverses to~$\vvv \vv\inv$.  In the complemented case and when there exists a finite set of words that is closed under complement, Proposition~\ref{P:Terminate} implies
\begin{equation}
\label{E:QuadraticBis}
\compl(\ww, \www) \le \CC \cdot \vert \ww\vert \cdot \vert \www\vert,
\end{equation}
for all words~$\ww, \www$. This holds for every lcm-presentation of a Garside monoid, \ie, every presentation obtained by selecting, for each pair of minimal generators~$\ss, \sss$, words~$\vv, \vvv$ such that both $\ss \vvv$ and~$\sss \vv$ represent the right-lcm of~$\ss$ and~$\sss$. So, \eqref{E:QuadraticBis} holds in particular for the standard presentation of the spherical Artin--Tits groups and, even more particularly, for Artin's presentation of the braid group~$B_\nn$: for each fixed~$\nn$, there exists a constant~$\CC_\nn$ such that $\compl(\ww, \www) \le\nobreak \CC_\nn \ell^2$ holds for all positive $\nn$-strand braid words of length at most~$\ell$.

Things become much more difficult when we go to~$B_\infty$, \ie, we impose no fixed limit on the indices of the letters~$\sig\ii$. 

\begin{ques}
\label{Q:BraidComplexity}
What is the least upper bound for the reversing complexity of~$(\ww, \www)$ for $\ww, \www$ positive braid words of length at most~$\ell$?
\end{ques}

Surprisingly, the answer is not known. The inequality
$$\compl(\sig1 \sig3 ... \sig{2\ell-1}, \sig{2\ell} \sig{2\ell-2} ... \sig2) \ge \frac43 \ell^4$$
is established in~\cite{Dhw}, and we conjecture that $O(\ell^4)$ is the highest possible complexity. On the other hand, the only upper bounds proved so far are an exponential bound~$O(3^{4\ell})$ in~\cite{Dhg}, improved to~$O(3^{2\ell})$ in~\cite{AutT}, both requiring rather delicate arguments.

Let us also recall that, in the infinitary context mentioned in Remark~\ref{R:LD}, the only proved upper bound for~$\compl(\ww, \ww)$ is a tower of exponentials of exponential height with respect to the length of~$\ww$ and~$\www$---which is not superseded by the double exponential of Proposition~\ref{P:DoubleExp} as, here, we do not assume the initial words to be equivalent.

\subsection{Lower bounds}
\label{S:Lower}

We shall conclude this survey with another application of subword reversing, namely an application to establishing lower bounds on the combinatorial distance.

The problem we address is to establish effective lower bounds on the combinatorial distance between two words, usually a difficult task. By contrast, explicitly computing the reversing distance may be relatively easy when we consider words of a particular form. The problem is that, in order to deduce from the value of~$\distRrev(\ww, \www)$ a lower bound on~$\dist(\ww, \www)$, we have to know that \eqref{E:Distances} is an equality. So our problem is to prove that reversing is optimal for some specific words, \ie, that the van Kampen diagram deduced from reversing entails as few tiles as possible. We shall now describe a method for answering such questions in the case of Artin's presentation of braid groups~\eqref{E:BraidPres}. This method is reminiscent of the approach developed in~\cite{Dhw} for establishing lower bounds on the rotation distance between binary trees. 

By definition, a van Kampen diagram consists of tiles, each of which is indexed by some relation of the presentation.  In order to prove that a van Kampen diagram~$\KKK$ is possibly optimal, we can attribute names to the tiles and, typically, show that any van Kampen diagram for the considered pair of words must contain a certain number~$N_1$ of tiles with name~$\nu_1$, plus a certain number~$N_2$ of tiles with name~$\nu_2$, etc. If the total number of tiles in~$\KKK$ is the sum of the various numbers~$N_1, N_2$, ..., we are sure that $\KKK$ is optimal.

In the current case of braids, we shall use a ``name vs.\,position'' duality to attribute names to the edges of van Kampen diagrams and, from there, to the tiles. It is standard---see for instance~\cite{ElM} or~\cite{Eps}---to associate with each positive braid word~$\ww$, \ie, every sequence of letters~$\sig\ii$, a \emph{braid diagram}~$\Diag\ww$ consisting of strands that cross, so that $\sig\ii$ corresponds to a crossing of the strands at positions~$\ii$ and~$\ii + 1$ (Figure~\ref{F:Sigma}). 

\begin{figure}[htb]
\begin{picture}(50,9)
\put(0,0){\includegraphics{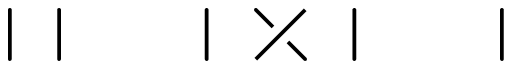}}
\put(-8,2){$\sig\ii$:}
\put(-1,7){$1$}
\put(4,7){$2$}
\put(24,7){$\ii$}
\put(27,7){$\ii\!+\!1$}
\put(49,7){$\nn$}
\put(11,2){...}
\put(41,2){...}
\end{picture}
\caption{\sf The $\nn$-strand braid diagram associated with~$\sig\ii$; for an arbitrary braid word~$\ww$, the diagram~$\Diag\ww$ is obtained by stacking one above the other the diagrams corresponding to the successive letters.}
\label{F:Sigma}
\end{figure}

Each strand in a braid diagram has a well-defined initial position, hereafter called its \emph{name}, and we can associate with each crossing of the diagram, hence with each letter in the braid word that encodes it,  the names of the strands involved in the crossing. As two strands may  cross more than once, we shall also include the rank of the crossing, thus using the name $\Name\pp\qq\aa$ for the $\aa$th crossing of the strands with initial positions~$\pp$ and~$\qq$. In this way, we associate with each positive braid word a sequence of names and, from there, we attribute names to the edges in any (braid) van Kampen diagram.

\begin{defi}[\bf name]
(Figure~\ref{F:BraidNames})
Let $\ee$ be an edge in a van Kampen diagram~$\KKK$ for~$\BP\nn$. Let $\ww$ be the braid word encoding a path~$\gamma$ that connects the source vertex of~$\KKK$ to the source vertex of~$\ee$. Then the \emph{name} of~$\ee$ is defined to be $\Name\pp\qq\aa$, where $\pp$ and~$\qq$ are the initial positions of the strands that finish at position~$\ii$ and~$\ii+1$ in the braid diagram associated with~$\ww$ and $\aa-1$ is the number of times the latter strands cross in this diagram. 
\end{defi}

\begin{figure}[htb]
\begin{picture}(64,22)(0,0)
\put(0,0.3){\includegraphics{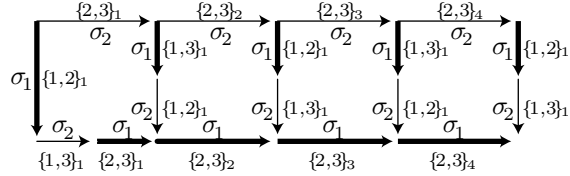}}
\put(6,19){$\sName231$}
\put(8,16){$\sig2$}
\put(22,19){$\sName232$}
\put(24,16){$\sig2$}
\put(38,19){$\sName233$}
\put(40,16){$\sig2$}
\put(54,19){$\sName234$}
\put(56,16){$\sig2$}
\put(-2.5,10){$\sig1$}
\put(1.5,10){$\sName121$}
\put(13.5,14){$\sig1$}
\put(17.5,14){$\sName131$}
\put(13.5,6){$\sig2$}
\put(17.5,6){$\sName121$}
\put(29.5,14){$\sig1$}
\put(33.5,14){$\sName121$}
\put(29.5,6){$\sig2$}
\put(33.5,6){$\sName131$}
\put(45.5,14){$\sig1$}
\put(49.5,14){$\sName131$}
\put(45.5,6){$\sig2$}
\put(49.5,6){$\sName121$}
\put(61.5,14){$\sig1$}
\put(65.5,14){$\sName121$}
\put(61.5,6){$\sig2$}
\put(65.5,6){$\sName131$}
\put(3,3.5){$\sig2$}
\put(1,-1){$\sName131$}
\put(11,3.5){$\sig1$}
\put(9,-1){$\sName231$}
\put(23,3.5){$\sig1$}
\put(21,-1){$\sName232$}
\put(39,3.5){$\sig1$}
\put(37,-1){$\sName233$}
\put(55,3.5){$\sig1$}
\put(53,-1){$\sName234$}
\end{picture}
\caption{\sf Attributing names to the edges in a braid van Kampen diagram (here a reversing diagram); for instance, the rightmost horizontal $\sig1$ edge on the bottom line receives the name $\Name234$ because, when one starts from the top left vertex, it corresponds to the fourth crossing of the strands that start at positions~$2$ and~$3$.}
\label{F:BraidNames}
\end{figure}

It is easy to check that the name of the edge~$\ee$ does not depend on the choice of the path~$\gamma$. The following relations immediately follow from the geometric meaning of the names and from the interpretation of~$\sig\ii$ in terms of strand crossing.

\begin{lemm}
\label{L:Names}
Assume that $\KKK$ is a van Kampen diagram for~$\BP\nn$, and $\ff$ is a face of~$\KKK$. If $\ff$ is a hexagon, \ie, if $\ff$ corresponds to a relation $\sig\ii\sig\jj \sig\ii = \sig\jj\sig\ii \sig\jj$ with $\vert\ii - \jj\vert = 1$, there exist pairwise distinct numbers~$\pp, \qq, \rr$ in~$\{1, ..., \nn\}$ and integers~$\aa, \bb, \cc$ such that the names of the edges bounding~$\ff$ respectively are
\begin{equation}
\label{E:Names1}
(\Name\pp\qq\aa, \Name\pp\rr\bb, \Name\qq\rr\cc)
\quad\mbox{and}\quad
(\Name\qq\rr\cc, \Name\pp\rr\bb, \Name\pp\qq\aa).
\end{equation}
Similarly, if $\ff$ is a square, \ie, if $\ff$ corresponds to a relation $\sig\ii\sig\jj = \sig\jj\sig\ii $ with $\vert\ii - \jj\vert \ge2$, there exist pairwise distinct numbers~$\pp, \qq, \rr, \ss$ in~$\{1, ..., \nn\}$ and integers~$\aa, \bb$ such that the names of the edges bounding~$\ff$ respectively are
\begin{equation}
\label{E:Names2}
(\Name\pp\qq\aa, \Name\rr\ss\bb)
\quad\mbox{and}\quad
(\Name\rr\ss\bb, \Name\pp\rr\aa).
\end{equation}
\end{lemm}

The proof is essentially contained in the diagrams of Figure~\ref{F:Separatrices}. Now comes a first optimality criterion.

\begin{prop}
\label{P:Optimal}
Call a family of names \emph{sparse} if it contains no name of the form~$\Name\pp\rr\cc$ whenever it contains $\Name\pp\qq\aa$ and $\Name\qq\rr\bb$. Then every van Kampen diagram~$\KKK$ with the property that there exists a sparse family~$\FF$ such that each face of~$\KKK$ entails exactly two names from~$\FF$ is optimal.
\end{prop}

\begin{proof}
Assume that $\KKK$ is a van Kampen diagram for~$(\ww, \www)$. Let $(\ww_0, \, ... \, , \ww_\mm)$ be a derivation from~$\ww$ to~$\www$ associated with~$\KKK$ as in Lemma~\ref{L:Derivation}. Let $\SS(\ww_\ii)$ be the sequence formed by the names of the successive letters of~$\ww_\ii$, and $\SS_\FF(\ww_\ii)$ be the subsequence of~$\SS(\ww_\ii)$ obtained by deleting all names that do not belong to~$\FF$.

By construction, the words~$\ww_\ii$ and~$\ww_{\ii+1}$ differ by exactly one braid relation, and the explicit formulas~\eqref{E:Names1} and~\eqref{E:Names2} imply that the sequence~$\SS(\ww_{\ii+1})$ is obtained from the sequence~$\SS(\ww_\ii)$ by reversing either a triple of names, or a pair of names. Moreover, under the assumption of the proposition, the sequence~$\SS_\FF(\ww_{\ii+1})$ is obtained from the sequence~$\SS_\FF(\ww_\ii)$ by reversing exactly one pair of names in every case. Therefore, the number of inversions between~Ê$\SS_\FF(\ww)$ and~$\SS_\FF(\www)$ is~$\mm$.

The hypothesis that $\FF$ is sparse implies that one braid relation can cause at most one inversion in an $\SS_\FF$~sequence (whereas it may cause three inversions in an $\SS$~sequence). Therefore, it is impossible to go from~$\ww$ to~$\www$ by using less that $\mm$~relations. In other words, we have $\dist(\ww, \www) = \mm$.
\end{proof}

Before giving examples, we reformulate the criterion of Proposition~\ref{P:Optimal} in more geometric terms. The formulas of Lemma~\ref{L:Names} show that, in every face of a braid van Kampen diagram, the same names occur on both sides, but in reversed order, as shown in Figure~\ref{F:Separatrices}. For each name~$\Name\pp\qq\aa$ occurring in~$\KKK$, connecting the middles of the edges with that name provides a curve, hereafter denoted~$\SEP\pp\qq\aa$, which is transversal to the edges of the diagram. Such curves are similar to the \emph{separatrices} of~\cite{Dhx} (which correspond to the special case of so-called simple braids), and we shall use the same terminology here. Then the geometric meaning of~\eqref{E:Names1} and~\eqref{E:Names2} is then that, in each hexagon, three separatrices cross each other whereas, in each square, two separatrices cross.

\begin{figure}[htb]
\begin{picture}(81,30)(0,0)
\put(0,0.3){\includegraphics{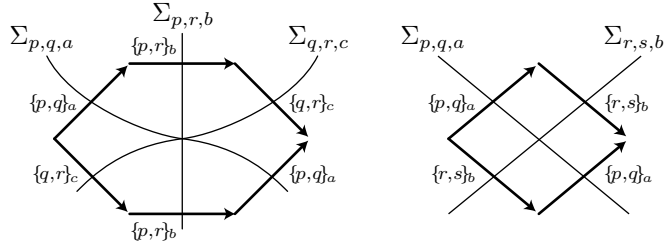}}
\put(-1.5,16.5){$\sName\pp\qq\aa$}
\put(-1,7){$\sName\qq\rr\cc$}
\put(-4,25){$\Sep_{\pp, \qq, \aa}$}
\put(12,24){$\sName\pp\rr\bb$}
\put(12,-0.5){$\sName\pp\rr\bb$}
\put(15,28){$\Sep_{\pp, \rr, \bb}$}
\put(33,16.5){$\sName\qq\rr\cc$}
\put(33,7){$\sName\pp\qq\aa$}
\put(33,25){$\Sep_{\qq, \rr, \cc}$}
\put(51.5,16.5){$\sName\pp\qq\aa$}
\put(52,7){$\sName\rr\ss\bb$}
\put(48,25){$\Sep_{\pp, \qq, \aa}$}
\put(75,16.5){$\sName\rr\ss\bb$}
\put(75,7){$\sName\pp\qq\aa$}
\put(75,25){$\Sep_{\rr, \ss, \bb}$}
\end{picture}
\caption{\sf Separatrices in a van Kampen diagram for~$\BP\nn$: applying one braid relation reverses the sequence of names of the edges, so, by connecting the edges with the same name, we obtain curves, called separatrices, that cross in the middle of the face.}
\label{F:Separatrices}
\end{figure}

In this context, Proposition~\ref{P:Optimal} can be reformulated in the language of separatrices. If $\FF$ is a family of names, we naturally say that a separatrix is an \emph{$\FF$-separatrix} if it corresponds to a name belonging to~$\FF$. Owing to the subsequent applications, we state the result for a reversing diagram.  

\begin{coro}
\label{C:Optimal}
Assume that $\ww, \www$ are positive braid words and there exists a sparse family of names~$\FF$ such that each face of the reversing diagram for~$\ww\inv \www$ contains exactly one crossing of $\FF$-separatrices, and any two $\FF$-separatrices cross at most once in that diagram. Then reversing is optimal for~$(\ww, \www)$.
\end{coro}

\begin{exam}
\label{X:Optimal}
Reversing the braid word $(\sig2 \sigg12 \sig2)^{-\mm} \sigg1{2\mm}$ leads to the equivalent braid words $(\sig2 \sigg12 \sig2)^\mm \sigg1{2\mm}$ and $\sigg1{2\mm} (\sig2 \sigg12 \sig2)^\mm$, as shown in Figure~\ref{F:Optimal1}. Let $\FF$ consists of the names~$\Name12\aa$ and~$\Name23\bb$. Then $\FF$ is sparse, and the diagram of Figure~\ref{F:Optimal1} satisfies the requirements of Corollary~\ref{C:Optimal}. Hence this diagram is an optimal van Kampen diagram, \ie, we have
$$\dist((\sig2 \sigg12 \sig2)^\mm \sigg1{2\mm}, \sigg1{2\mm} (\sig2 \sigg12 \sig2)^\mm) = 4\mm^2.$$
This gives a short proof of the result of~\cite{HKN}.
\end{exam} 

\begin{figure}[htb]
\begin{picture}(66,72)(0,0)
\put(0,0){\includegraphics{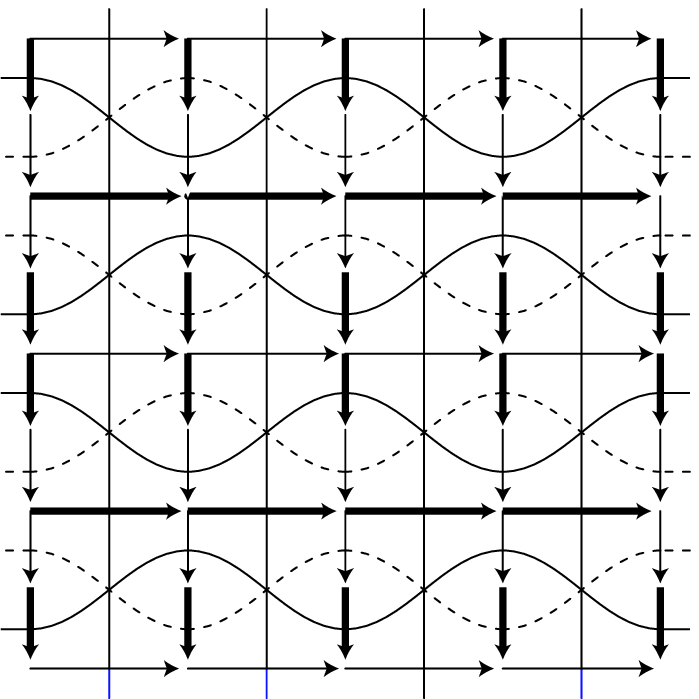}}
\put(8,71){$\sName121$}
\put(24,71){$\sName122$}
\put(40,71){$\sName123$}
\put(56,71){$\sName124$}
\put(-7,63){$\sName231$}
\put(-7,55){$\sName131$}
\put(-7,47){$\sName132$}
\put(-7,39){$\sName232$}
\put(-7,31){$\sName233$}
\put(-7,23){$\sName133$}
\put(-7,15){$\sName134$}
\put(-7,7){$\sName234$}
\end{picture}
\caption{\sf Reversing diagram for the braid words of Example~\ref{X:Optimal} (here with $\mm = 2$). Thin edges represent~$\sig1$, thick edges represent~$\sig2$. The useful separatrices are thin plain lines: each hexagon contains one crossing of such lines, and any two of them cross at most one (we ignore the separatrices with name $\Name13\cc$, drawn in dotted line). By Corollary~\ref{C:Optimal}, the diagram is optimal, \ie, it achieves the combinatorial distance.}
\label{F:Optimal1}
\end{figure}

Another example is shown in Figure~\ref{F:Optimal2}. Here one simply starts with the braid words~$\sigg1{2\mm}$ and $\sigg2{2\mm}$, and the conclusion is again that reversing is optimal; here also, we consider the separatrices with names~$\Name12\aa$ and~$\Name23\bb$.

\begin{figure}[htb]
\begin{picture}(72,72)(0,0)
\put(0,0){\includegraphics{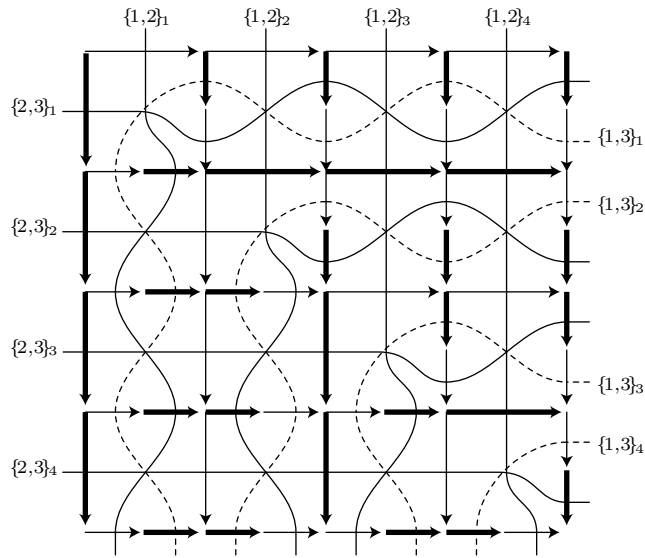}}
\put(8,71){$\sName121$}
\put(24,71){$\sName122$}
\put(40,71){$\sName123$}
\put(56,71){$\sName124$}
\put(-7,59){$\sName231$}
\put(-7,43){$\sName232$}
\put(-7,27){$\sName233$}
\put(-7,11){$\sName234$}
\put(71,55){$\sName131$}
\put(71,46){$\sName132$}
\put(71,22){$\sName133$}
\put(71,14){$\sName134$}
\end{picture}
\caption{\sf Reversing from $\sigg2{-2\mm} \sigg1{2\mm}$ is optimal. Here again, we consider the separatrices with names $\Name12\aa$ and $\Name23\bb$, and forget about those with name~$\Name13\cc$.}
\label{F:Optimal2}
\end{figure}

The above optimality results are quite partial since they only involve the very specific case of Artin--Tits braid monoids. We refer to~\cite{Dhx} for further results, in a case that is still more restricted, namely that of simple braid words, \ie, positive braid words corresponding to braid diagrams in which any two strands cross at most once. In this case, which is equivalent to the case of reduced decompositions of permutations into products of transpositions, the names are all of the form~$\Name\pp\qq1$ and simple optimality criteria can be stated: for instance, the hypothesis that any two separatrices cross at most once guarantees optimality. An interesting feature is that, in some results, the \emph{metric} aspects of the reversing diagrams---as opposed to their topological aspects---play a crucial role.

\section{Conclusion}

In good cases, namely for complete presentations, subword reversing can be used to investigate a presented semigroup and its possible group of fractions, mainly to prove cancellativity, to solve word problems, to recognize specific families such as Garside structures, to compute in such structures, possibly to obtain optimal derivations. It seems reasonable to hope for more applications in the future. 

A last comment is in order. Once completeness is granted, using words and reversing is essentially equivalent to using elements of the monoid and common multiples. However, before completeness is established,  it is crucial to distinguish between words and the elements they represent: reversing equivalent words need not lead to equivalent results in general, and subword reversing is really an operation on words, which in general makes no sense at the level of the elements of the associated semigroup or group.


\begin{thebibliography}{99}

\def\Reff#1; #2; #3; #4; #5; #6; #7\par{%
\bibitem{#1} #2, \emph{#3}, #4 {\bf #5} (#6) #7}

\def\Ref#1; #2; #3; #4\par{%
\bibitem{#1} #2, \emph{#3}, #4}

\Reff Adj; S.I. Adyan; On the embeddability of monoids; Soviet. Math. Dokl.; 1-4; 1960; 819--820.

\Reff AdjNF; S.I.\,Adyan; Fragments of the word Delta in a braid group; Mat. Zam. Acad. Sci. SSSR; 36-1; 1984; 25--34; translated Math. Notes of the Acad. Sci. USSR; 36-1 (1984) 505--510.

\Ref Aut; M.\,Autord; Comparing Gr\"obner bases and word reversing; math.0712.0525, Southeast Asian Bull. Math., to appear.

\Ref AutT; M.\,Autord; Aspects algorithmiques du retournement de mot; PhD Thesis, Universit\'e de Caen; 2009.

\Ref Dhx; M.\,Autord \& P.\,Dehornoy; On the distance between the expressions of a permutations; arXiv: math.CO/0902.3074.

\Ref Bes; D.\,Bessis; Garside categories, periodic loops and cyclic sets; math.GR/0610778.

\Reff BrS; E. Brieskorn \& K. Saito; Artin-Gruppen und Coxeter-Gruppen; Invent. Math.; 17; 1972; 245--271.

\Reff CMW; R.\,Charney, J.\,Meier \& K.\,Whittlesey; Bestvina's normal form complex and the homology of Garside groups; Geom. Dedicata; 105; 2004; 171-188.

\Ref Cho; F.\,Chouraqui; Garside groups and Yang--Baxter equations; Comm. Algebra; to appear.

\Ref ClP; A.H.\,Clifford \& G.B.\,Preston; The algebraic Theory of Semigroups, vol.~1; Amer. Math. Soc. Surveys {\bf 7}, (1961).

\Reff Cor; R.\,Corran; A normal form for a class of monoids including the singular braid monoids; J. Algebra; 223; 2000; 256--282.

\Reff Dez; P.\,Dehornoy; Preuve de la conjecture d'irr\'eflexivit\'e  pour les structures distributives libres;  C. R. Acad. Sci. Paris; 314; 1992; 333--336.

\Reff Dfa; P.\,Dehornoy; Deux propri\'et\'es des groupes de tresses; C. R. Acad. Sci. Paris; 315; 1992; 633--638.

\Reff Dfb; P.\,Dehornoy; Braid groups and left distributive  operations; Trans. Amer. Math. Soc.; 345-1; 1994; 115--151. 

\Reff Dff; P.\,Dehornoy; Groups with a complemented presentation; J. Pure Appl. Algebra; 116; 1997; 115--137.

\Reff Dfo; P.\,Dehornoy; A fast method for comparing braids; Advances in Math.; 125; 1997; 200--235.

\Reff Dgc; P.\,Dehornoy; On completeness of word reversing; Discrete Math.; 225; 2000; 93--119.

\Ref Dgd; P.\,Dehornoy; Braids and Self-Distributivity; Progress in Math. vol. 192, Birkh\"auser (2000).

\Reff Dgk; P.\,Dehornoy; Groupes de Garside;  Ann. Scient. Ec. Norm. Sup.; 35; 2002; 267--306.

\Reff Dgp; P.\,Dehornoy; Complete positive group presentations; J.  Algebra; 268;
2003; 156--197.

\Ref Dhw; P.\,Dehornoy; On the rotation distance between binary trees; Advances in Math., to appear; math.CO/0901.2557.

\Reff Dht; P.\,Dehornoy; Left-Garside categories, self-distributivity, and braids; Ann. Math. Blaise Pascal; 16; 2009; 189--244.

\Reff Dgl; P.\,Dehornoy \& Y. Lafont; Homology of Gaussian groups;  Ann. Inst. Fourier; 53-2; 2003; 1001--1052.

\Reff Dfx; P.\,Dehornoy \& L. Paris; Gaussian groups and Garside groups, two generalizations of Artin groups; Proc. London Math. Soc.; 79-3; 1999; 569--604.

\Reff Dhg; P.\,Dehornoy \& B.\,Wiest; On word reversing in braid groups; Int. J. Algebra Comput.; 16(5); 2006; 931--947.

\Reff Dlg; P.\,Deligne; Les immeubles des groupes de tresses g\'en\'eralis\'es; Invent. Math.; 17; 1972; 273--302.

\Ref DiM; F.\,Digne \& J.\,Michel; Garside and locally Garside categories; math.GR/0612652.

\Reff ElM; E. A. Elrifai \& H. R. Morton; Algorithms for positive braids; Quart. J. Math. Oxford; 45-2; 1994; 479--497.

\Ref Eps; D.\,Epstein, with J.\,Cannon, D.\,Holt, S.\,Levy, M.\,Paterson \& W.\,Thurston; Word Processing in Groups; Jones \& Bartlett Publ. (1992).

\Reff FrG; N.\,Franco \& J.\,Gonz\'alez-Meneses; Conjugacy problem for braid groups and Garside groups; J. Algebra; 266-1; 2003; 112--132.

\Ref GarT; F.A.\,Garside; The theory of knots and associated problems; PhD thesis, Oxford (1965). 

\Reff Gar; F.A.\,Garside; The braid group and other groups; Quart. J. Math. Oxford; 20-78; 1969; 235--254.

\Reff Geb; V.\,Gebhardt; A new approach to the conjugacy problem in Garside groups; J. Algebra; 292-1; 2005; 282--302.

\Ref GeG; V.\,Gebhardt \& J.\,Gonz\'alez-Meneses; The cyclic sliding operation in Garside groups; Math. Zeitschr., to appear. 
 
\Ref HKN; J.\,Hass, A.\,Kalka, and T.\,Nowik; Complexity of relations in the braid group; math.GR/0906.0137.

\Ref Hum;  J.E.\,Humphreys; Reflection Groups and Coxeter Groups; Cambridge Univ. Texts; 1989.

\Reff Krb; D.\,Krammer; Braid groups are linear; Ann. Math.; 151-1; 2002; 131--156.

\Reff Kra; D.\,Krammer; A class of Garside groupoid structures on the pure braid group; Trans. Amer. Math. Soc.; 360; 2008; 4029-4061.

\Ref LyS; R.C.\,Lyndon and P.E.\,Schupp; Combinatorial Group Theory; Springer-Verlag; 1977, reprinted in 2001.

\Ref McC; J.\,McCammond; An introduction to Garside structures; Preprint (2005).

\Reff Par; L.\,Paris; Artin monoids inject in their groups;  Comment. Math. Helv.; 77; 2002; 609-637. 

\Reff Pim; M.\,Picantin; Garside monoids vs. divisibility monoids; Math. Struct. in Comp. Sci.; 15-2; 2005; 231-242.

\Reff Rem;  J.H.\,Remmers; On the geometry of semigroup presentations; Advances in Math.; 36; 1980; 283--296.

\Reff Squ; C. Squier; The homological algebra of Artin groups; Math. Scand.; 75; 1995; 5--43.

\Reff Tat; K.\,Tatsuoka; An isoperimetric inequality for Artin groups of finite type; Trans. Amer. Math. Soc.; 339--2; 1993; 537--551.

\end{thebibliography}
\end{document}